\documentclass[11pt]{amsart}
\newtheorem{theorem}{Theorem}[section]
\newtheorem{lemma}[theorem]{Lemma}

\newtheorem{convention}[theorem]{Convention}
\theoremstyle{definition}
\newtheorem{definition}[theorem]{Definition}

\theoremstyle{remark}

\newtheorem*{remark*}{Remark}

\newtheorem{claim}{Claim}

\newcommand{\abs}[1]{\lvert#1\rvert}

\usepackage{amsmath,amsthm,verbatim,amssymb,amsfonts,amscd,graphicx,bm,fixmath}
\usepackage{graphics}
\usepackage[
inner = 1.3in,
outer = 1.3in,
top = 0.9in,
bottom = 0.9in,
]{geometry}

\usepackage{microtype} 
\usepackage{xcolor}
\usepackage{enumitem} 
\setlist{nolistsep} 

\usepackage{tikz} 
\usepackage{soul}
\usepackage{pstricks,pst-plot,pst-node}
\usepackage{graphicx}
\usepackage{hyperref}

\usetikzlibrary{arrows, positioning, automata}
\newcommand\bD{\chi}

\newcommand\cT{\mathcal{T}}

\newcommand\cG{\mathcal{G}}

\newcommand\cS{\mathcal{S}}

\newcommand\bU{\mathbf{U}}
\newcommand\bV{\mathbf{V}}
\newcommand\bW{\mathbf{W}}

\newcommand\eU{E_\mathbf{U}}
\newcommand\eV{E_\mathbf{V}}
\newcommand\eW{E_\mathbf{W}}

\newcommand\concat{\symbol{94}}

\newcommand{\bigsame}{\mathsf{same}}

\newcommand{\Permit}{\mathop{\mathrm{Permit}}}
\newcommand\res{{\,\upharpoonright\,}}

\newcommand{\loa}{\ell}

\newcommand\nbd{\nobreakdash-\hspace{0pt}}
\DeclareMathOperator{\tp}{q}

\DeclareMathOperator{\type}{candidate}
\DeclareMathOperator{\per}{h}
\DeclareMathOperator{\cf}{CF}
\DeclareMathOperator{\dw}{dw}
\DeclareMathOperator{\test}{test}
\DeclareMathOperator{\Test}{Test}
\DeclareMathOperator{\visit}{\mathsf{visit}}

\begin{document}
\author[Cai]{Mingzhong Cai}
\address[Cai]{Hyperimmune Books\\
Suwanee, GA 30024\\
USA}
\email{\href{mailto:mingzhongcai@gmail.com}{mingzhongcai@gmail.com}}

\author[Yiqun Liu]{Yiqun Liu}
\address[Yiqun Liu]{Office of the President\\
National University of Singapore\\
Singapore 119077\\
SINGAPORE}
\email{\href{mailto:liuyq@nus.edu.sg}{liuyq@nus.edu.sg}}

\author[Yong Liu]{Yong Liu}
\address[Yong Liu]{School of Information Engineering\\
Nanjing Xiaozhuang University\\
CHINA}
\email{\href{mailto:liuyong@njxzc.edu.cn}{liuyong@njxzc.edu.cn}}

\author[Peng]{Cheng Peng}
\address[Peng]{Department of Mathematics\\
Hebei University of Technology\\
CHINA}
\email{\href{mailto:pengcheng@hebut.edu.cn}{pengcheng@hebut.edu.cn}}

\author[Yang]{Yue Yang}
\address[Yang]{Department of Mathematics\\
National University of Singapore\\
Singapore 119076\\
SINGAPORE}
\email[Yang]{\href{mailto:matyangy@nus.edu.sg}{matyangy@nus.edu.sg}}

\subjclass[2020]{03D25}

\keywords{r.e.~degrees, minimal pair}

\thanks{Peng’s research was partially supported by NSF of China No. 12271264.
Yang’s research was partially supported by NUS grant WBS: R-146-000-337-114}

\begin{abstract}
Two nonzero recursively enumerable (r.e.) degrees $\mathbold{a}$ and $\mathbold{b}$ form a strong minimal pair if $\mathbold{a}\wedge \mathbold{b}=\mathbf{0}$ and $\mathbold{b}\vee \mathbold{x}\geq \mathbold{a}$ for any nonzero r.e.~degree $\mathbold{x}\leq \mathbold{a}$.
We prove that there is no strong minimal pair in the r.e.~degrees. 
Our construction goes beyond the usual \(\mathbf{0}'''\)\nbd{}priority arguments and we give some evidence to show that it needs \(\mathbf{0}^{(4)}\)\nbd{}priority arguments.
\end{abstract}

\title{On the nonexistence of a strong minimal pair}
\maketitle

\section{Introduction}

In paper~\cite{Barmpalias.Cai.ea:2015}, Barmpalias, Cai, Lempp and Slaman claimed the existence of a strong minimal pair.
Two nonzero recursively enumerable (r.e.) degrees $\mathbold{a}$ and $\mathbold{b}$ form a \emph{strong minimal pair} if
$\mathbold{a}\wedge \mathbold{b}=\mathbf{0}$ and for any nonzero r.e.~degree $\mathbold{x}\leq \mathbold{a}$,
$\mathbold{b}\vee \mathbold{x}\geq \mathbold{a}$. 
The notion of strong minimal pairs can also be viewed as a strengthening of the so-called ``Slaman triples''.  
Recall that three r.e.~degrees $\mathbold{a}, \mathbold{b}$ and $\mathbold{c}$ form a \emph{Slaman triple} if $\mathbold{a} \neq \mathbf{0}$, $\mathbold{c}\nleq \mathbold{b}$ 
and for any nonzero r.e.~degree $\mathbold{x}\leq \mathbold{a}$, $\mathbold{b}\vee \mathbold{x}\geq \mathbold{c}$. 
In a Slaman triple formed by \(\mathbold{a}\), \(\mathbold{b}\) and \(\mathbold{c}\), it is clear that \(\mathbold{a}\) and \(\mathbold{b}\) form a minimal pair; if \(\mathbold{a}\) and \(\mathbold{c}\) coincide, then \(\mathbold{a}\) and \(\mathbold{b}\) will form a strong minimal pair.

We refer the readers to~\cite{Barmpalias.Cai.ea:2015} for the significance of this study,
as well as for its ``long and twisted'' history.  Their paper was devoted to show the existence of a strong minimal pair,
where they gave a detailed illustration of how to combine two sets of requirements before proceeding to the long and complicated full construction.
They also raised the open question whether there exists a ``two-sided'' strong minimal pair---a strong minimal pair as above
but with the extra clause ``and for any nonzero r.e.~degree $\mathbold{y}\leq \mathbold{b}$, $\mathbold{a}\vee \mathbold{y}\geq \mathbold{b}$''?

While trying to settle the existence of a two-sided strong minimal pair, we encountered a serious difficulty of dealing with three sets of requirements.  It turns out that their original construction cannot overcome this difficulty either.
Eventually, there is another twist in the history of the problem.  By employing three families of r.e.~sets, 
we are able to establish the \emph{nonexistence} of a strong minimal pair:
\begin{theorem} [Main]\label{thm:main}
  For any r.e.~degrees $\mathbold{a}$ and $\mathbold{b}$, either $\mathbold{a}\leq \mathbold{b}$ or
  there exists an r.e.~degree $\mathbold{x}\leq \mathbold{a}$ such that $\mathbold{x}\neq \mathbf{0}$
  and $\mathbold{x}\vee \mathbold{b}\not\geq \mathbold{a}$.
\end{theorem}

As an immediate consequence, the degrees $\mathbold{a}$ and $\mathbold{c}$ in a Slaman triple cannot coincide; in fact, $\mathbold{a}$ and $\mathbold{b}\vee\mathbold{c}$ also form a minimal pair. 
To see this, suppose toward a contradiction that there exists some nonzero r.e.~degree $\mathbold{x}$ below both $\mathbold{a}$ and $\mathbold{b}\vee\mathbold{c}$. Consider any nonzero r.e.~degree $\mathbold{w}$ below $\mathbold{x}$, we would have $\mathbold{b}\vee \mathbold{w}\geq \mathbold{c}$ and therefore  $\mathbold{b}\vee \mathbold{w}\ge\mathbold{b}\vee \mathbold{c}\ge \mathbold{x}$, which implies that $\mathbold{x}$ and $\mathbold{b}$ form a strong minimal pair, contradicting our main theorem.
This fact echoes the gap-cogap construction used in building a Slaman triple (see, for example, Shore and Slaman
\cite{Shore.Slaman:1993}), where elements are enumerated into $A$ and $B\cup C$ at alternating stages.

Besides the strong minimal pair problem itself, we are also interested in the techniques involved in the solution.
In the paper~\cite{Barmpalias.Cai.ea:2015}, a novel $c$\nbd{}outcome was introduced to handle some of the conflicts in a way that goes beyond $\mathbf{0}'''$-priority method.
In this paper, we use an $(\omega+1)$\nbd{}branching tree to organize our construction, furthermore, each $(\omega+1)$\nbd{}st branch can be considered as a gateway to a parallel $\Pi_3$\nbd{}world.
A similar
design has been used before, notably by Shore in~\cite{Shore:1988}.
However, the conflicts between our strategies seem more severe than those in Shore's. For instance, for certain two nodes on the priority tree,
one inside a $\Pi_3$\nbd{}world, the other outside, we cannot determine \emph{a priori} which node has higher priority; we have to assign priority dynamically:
whichever acts first will change the environment of the other.  Much more care has to be applied in showing the existence of the true path.
This brings up the question of what counts as a typical $\mathbf{0}^{(4)}$\nbd{}priority argument.
A $\mathbf{0}'''$\nbd{}priority argument can now be routinely presented in a framework where finite injuries happen on the true path.
Our construction seems to suggest that handling the interactions inside and outside the $\Pi_3$\nbd{}world must be a feature in any framework for a $\mathbf{0}^{(4)}$\nbd{}priority argument.
We hope that our work will bring us closer to a canonical framework for the $\mathbf{0}^{(4)}$\nbd{}priority argument.

\section{Toward the stage by stage construction}
\subsection{Preliminaries and conventions}
A \emph{Turing functional} \(\Gamma\) is an r.e.\ set of of triples, called the \emph{axioms}, \((\sigma,n,i)\in 2^{<\omega}\times \omega\times\{0,1\}\) such that if both \((\sigma,n,i)\), \((\tau,n,j)\in \Gamma\) and \(\sigma\preceq \tau\), then \(i=j\).
We write \(\Gamma(X;n)\downarrow = i\) if \(\exists l\ (X\res l,n,i)\in \Gamma\); in this case, we also define the \emph{use function} \(\gamma(X;n)\) to be the least such \(l\) plus 1. 
We write \(\Gamma(X;n)\uparrow\) if \(\forall l,i (X\res l, n, i)\notin \Gamma\). 
By assuming that Turing functionals are computing initial segments, we may further assume that the domain of a given Turing function is downward closed,
and its use function is nondecreasing with respect to its arguments.
For two subsets \(A\) and \(B\) of natural numbers, we write \(A\le_T B\) if there exists a Turing functional \(\Phi\) such that for each \(n\), \(A(n)=\Phi(B;n)\). Following the assumption above, whenever we write \(\Phi(B;n)=A(n)\), we mean \(\Phi(B;i)=A(i)\) for all \(i\le n\).
We will use \([s]\) to indicate that all calculations are restricted to a particular stage~\(s\), and in particular, all numerical parameters occurred are \(\leq s\).

For \(\Gamma(X)=A\), the length of agreement at stage~\(s\) is defined as
\[
    \loa(s)=\{n\le s\mid \forall k\le n, \Gamma(X;k)[s]=A(k)[s]\}.
\]
We say that \(s\) is an \emph{expansionary stage} if \(\loa(s)>\loa(s^*)\) for all $s^*<s$.

We define \(\bigsame(X,k,s,t)\) if and only if 
\[
(\forall n< k)(\forall s') [s\le s'\le t\rightarrow X(n)[s]=X(n)[s']].
\]
\(\bigsame(X,k,s,t)\) says that the set \(X\) up to the first \(k\) digits does not change from stage \(s\) to \(t\).

\subsection{Requirements}\label{sec:req}
Given r.e.\ sets \(A\) and \(B\) with \(A\nleq_T B\), we build an r.e.\ set \(X\) satisfying the following requirements:

The first one is the permitting requirement:
\begin{itemize}
\item \(\Permit(X)\): \(X\leq_T A\).
\end{itemize}

Fix an effective enumeration of Turing functionals \((\Gamma_e,\Delta_e)_{e\in \omega}\), we have the diagonalization requirements, here the letters $G$ and $D$ come from $\Gamma$ and $\Delta$ respectively:
\begin{itemize}
\item \(G_e(X)\): \(\Gamma_e(BX)\neq A\); and
\item \(D_e(X)\): \(\Delta_e\neq X\).
\end{itemize}

\(G_e(X)\) has three possible outcomes: Case (1), there is an \(n\) such that \(\Gamma(BX;n)\) never agrees with \(A(n)\); 
Case (2), there is a divergent point \(n\) such that \(\gamma(BX;n)[s]\to \infty\) as \(s\to \infty\); and Case (3), \(\Gamma_e(BX)=A\).

The strategy of \(G_e(X)\) is to check if the length of agreement between \(\Gamma(BX)\) and \(A\) goes to infinity.  
If the answer is no, then we win by the \(\Sigma_2\)\nbd{}outcome which is Case (1).  If the answer is yes, we have a \(\Pi_2\)\nbd{}outcome indicating either Cases (2) or (3). Then we try to build a Turing functional \(\Omega\) so that
\[
    \Gamma_e(BX)=A \Rightarrow \Omega(B)=A.
\]
In a usual priority argument, if \(\Omega(B)=A\) turns out to be partial, then we shall exhibit a divergent point of \(\Gamma_e(BX)\) as in Case (2); however, in our construction we might also have the possibility of Case (3) and we shall begin to work on \emph{another} set instead of \(X\)---we take a nonuniform approach.
In fact, instead of constructing a single r.e.\ set \(X\), we construct three types of candidates whose types are denoted by \(\bU\), \(\bV\) and \(\bW\). 
There will be a unique candidate for a \(\bU\)\nbd{}set, which will be denoted by \(U\). 
There will be countably many candidates each for \(\bV\)\nbd{}sets and \(\bW\)\nbd{}sets and will be denoted by \(V_\alpha\) or \(W_{\alpha,\beta}\) where the subscripts refer to their ``parent nodes''. We will argue that at least one of these candidates satisfies all requirements.

\subsection{Priority tree}
Fix a list of requirements as
\[
    G_0(X)<D_0(X)<G_1(X)<D_1(X)<\cdots,
\]
where \(X\) stands for a candidate that we build. Depending on whether \(X\) is a \(\bU\)\nbd{}set, \(\bV\)\nbd{}set or \(\bW\)\nbd{}set, the strategies will be different. 
Besides the nodes that work directly for our requirements, we also have other auxiliary nodes.

The priority tree \(\cT\) is defined recursively as follows. The root of \(\cT\) is assigned to \(\Permit(U)\). 
A node \(\alpha\) assigned to \(*\) will be called a \(*\)\nbd{}node. 
A node \(\alpha\) is said to be \emph{in the \(\Sigma_3\)\nbd{}world} if it is not declared to be in a \emph{\(\Pi_3\)\nbd{}world} (see~\ref{it:pt C} below).
For references, the first few nodes of the priority tree are given in Figure~\ref{fig:priority tree}, where we also omit some of the outcomes, particularly those with label ``\(0\)'', to save space. 
\begin{figure}
    \centering
    \includegraphics[width=\textwidth]{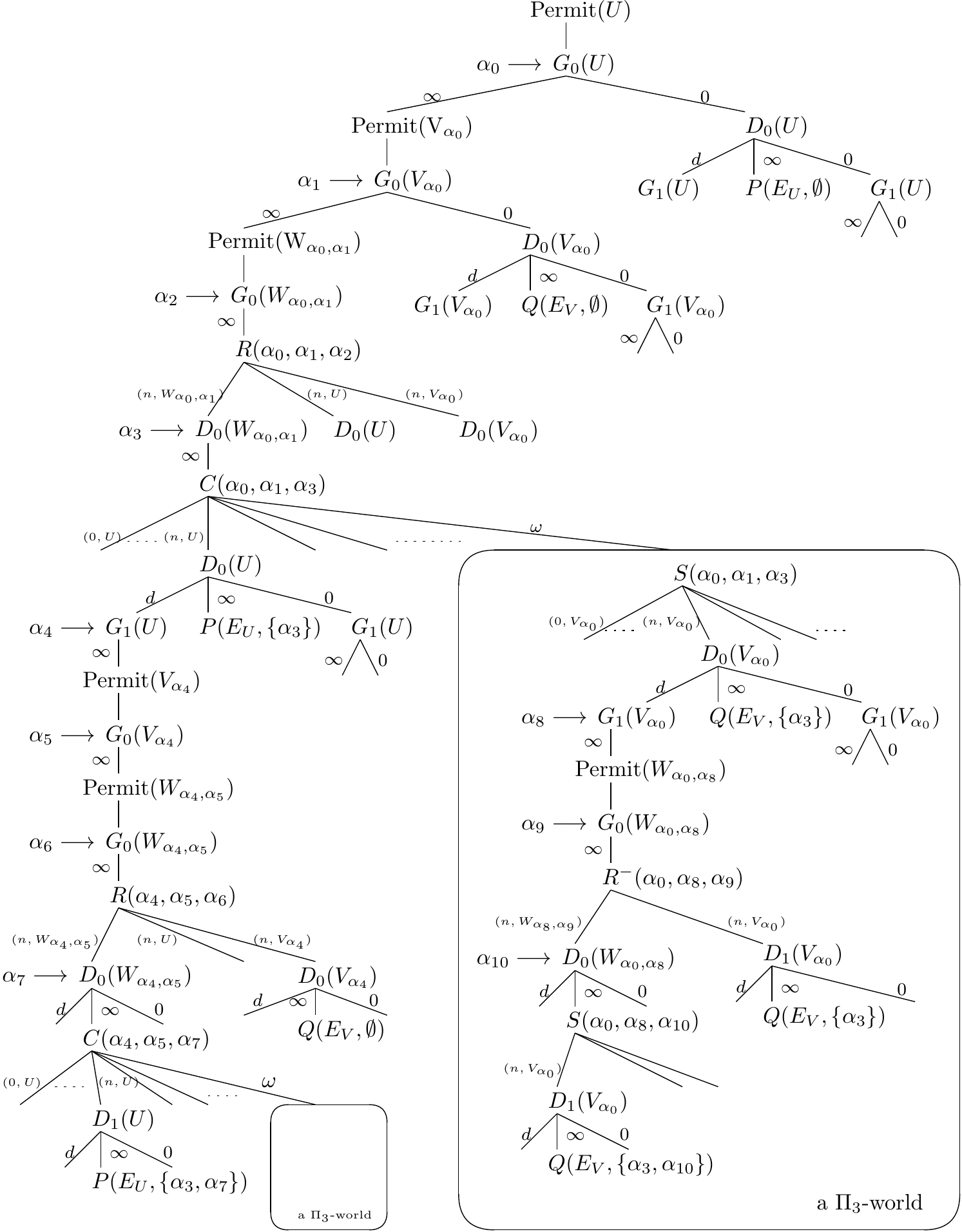}
    \caption{the priority tree}\label{fig:priority tree}
\end{figure}

Suppose that \(\alpha\) has been assigned.
\begin{enumerate}[label=(T\arabic*)]
    \item Suppose that \(\alpha\) is a \(\Permit(X)\)\nbd{}node, then \(\alpha\) has a unique outgoing edge labelled \(0\) and \(\alpha\concat 0\) is assigned to \(G_0(X)\).
\end{enumerate}
\begin{remark*}
    A \(\Permit(X)\)\nbd{}node is the place where we introduce a new candidate \(X\) and is the \emph{only} place where we enumerate numbers into \(X\). 
    The permitting is done in a delayed fashion, after a number gets permitted by $A$, its fate will be determined only by the next $\Permit(X)$-stage.
    After introducing the set $X$, we work on \(G_0(X)\) which is the first one in its requirement list.
\end{remark*}

\begin{enumerate}[resume*]
    \item Suppose that \(\alpha\) is a \(G_e(X)\)\nbd{}node, then it has two outgoing edges labelled \(\infty<_L 0\), here \(<_L\) means ``to the left of''.
    \begin{enumerate}[label=(T\arabic{enumi}.\arabic*)]
        \item\label{it:pt GU} If \(X=U\), then \(\alpha\concat \infty\) is assigned to \(\Permit(V_\alpha)\) and \(\alpha\concat 0\) to \(D_e(U)\).
        \item\label{it:pt GV} If \(X=V_{\alpha_0}\) for some \(\alpha_0\subseteq \alpha\), then \(\alpha\concat \infty\) is assigned to \(\Permit(W_{\alpha_0,\alpha})\) and \(\alpha\concat 0\) to \(D_e(V_{\alpha_0})\).
        \item\label{it:pt GW} If \(X=W_{\alpha_0,\alpha_1}\) for some \(\alpha_0\subseteq\alpha_1\subseteq \alpha\), 
        then \(\alpha\concat \infty\) is assigned to \(R(\alpha_0,\alpha_1,\alpha)\) if \(\alpha\) is also in the \(\Sigma_3\)\nbd{}world, and to \(R^-(\alpha_0,\alpha_1,\alpha)\) if \(\alpha\) is in a \(\Pi_3\)\nbd{}world; \(\alpha\concat 0\) is assigned to \(D_e(W_{\alpha_0,\alpha_1})\). 
    \end{enumerate}
\end{enumerate}


\begin{remark*}
    A \(G_e(X)\)\nbd{}node \(\alpha\) checks if the length of agreement of \(\Gamma_e(BX)=A\) goes to infinity: 
    if not, it has the \(\Sigma_2\)\nbd{}outcome \(0\) and \(G_e(X)\) is satisfied easily, and we move on to the next requirement \(D_e(X)\). 
    One may view \(\alpha\concat 0\) as a degenerated case, as the nodes extending \(\alpha\concat 0\) can safely ignore \(\alpha\). 
    In the case that \(G_e(X)\)\nbd{}node has the \(\Pi_2\)\nbd{}outcome \(\infty\), we first introduce more sets to form a three element group $(U, V_{\alpha_0}, W_{\alpha_0,\alpha_1})$.
    Only after we have \(X=W_{\alpha_0,\alpha_1}\) for some \(\alpha_0\) and \(\alpha_1\) with \(\alpha_0\concat \infty\subseteq \alpha_1\concat \infty\subseteq \alpha\concat \infty\), 
    we introduce an \(R(\alpha_0,\alpha_1,\alpha)\)\nbd{}node to detect a divergent point for
    $\Gamma$\nbd{}functionals at either \(\alpha_0\), \(\alpha_1\), or \(\alpha\). An \(R^-(\alpha_0,\alpha_1,\alpha)\)\nbd{}node is a degenerate case of $R$, where we detect a divergent point for $\Gamma$\nbd{}functionals at either \(\alpha_1\) or \(\alpha\).
    As one shall see, an $R^-$\nbd{}node necessarily lives in a $\Pi_3$\nbd{}world where the set $U$ is no longer considered.
\end{remark*}

\begin{enumerate}[resume*]
    \item Suppose that \(\alpha\) is a \(D_e(X)\)\nbd{}node, then \(\alpha\) has three outgoing edges labelled \(d<_L \infty<_L 0\).
    \begin{enumerate}[label=(T\arabic{enumi}.\arabic*)]
        \item\label{it:pt D-P} If \(X=U\), then \(\alpha\concat \infty\) is assigned to \(P(\eU,\eW)\) where \(E_\bU=\{\alpha\}\) and \(E_\bW\) is a finite set of nodes on the tree and will be defined later;
            \(\alpha\concat 0\) and \(\alpha\concat d\) are both assigned to \(G_{e+1}(U)\).
        \item\label{it:pt D-Q} If \(X=V_{\alpha_0}\), then \(\alpha\concat \infty\) is assigned to \(Q(\eV,\eW)\) where \(E_\bV=\{\alpha\}\) and \(E_\bW\) is a finite set of nodes on the tree and will be defined later; 
            \(\alpha\concat 0\) and \(\alpha\concat d\) are both assigned to \(G_{e+1}(V_{\alpha_0})\).
        \item\label{it:pt D-CS} If \(X=W_{\alpha_0,\alpha_1}\) and \(\alpha\) is in the \(\Sigma_3\)\nbd{}world, then \(\alpha\concat\infty\) is assigned to \(C(\alpha_0,\alpha_1,\alpha)\); 
        if \(X=W_{\alpha_0,\alpha_1}\) and \(\alpha\) is in a \(\Pi_3\)\nbd{}world, then \(\alpha\concat\infty\) is assigned to \(S(\alpha_0,\alpha_1,\alpha)\). 
        In both cases, \(\alpha\concat 0\) and \(\alpha\concat d\) are assigned to \(G_{e+1}(W_{\alpha_0,\alpha_1})\).
    \end{enumerate}
\end{enumerate}
\begin{remark*}
    A \(D_e(X)\)\nbd{}node \(\alpha\) first checks if it has some \(n\) such that \(\Delta_e(n)=0\) but \(X(n)=1\). 
    If so, \(\alpha\) has the \(\Sigma_1\)\nbd{}outcome \(d\) and \(D_e(X)\) is satisfied in the easiest way. 
    If not, \(\alpha\) checks if the length of agreement between \(\Delta_e\) and \(X\) goes to infinity. 
    The \(\Sigma_2\)\nbd{}outcome \(0\) indicates that the length of agreement stops increasing from some stage onward; the \(\Pi_2\)\nbd{}outcome \(\infty\) indicates the other case. 
    Below the \(\Pi_2\)\nbd{}outcome, we do not try to diagonalize against \(\Delta_e\) immediately if \(X=W_{\alpha_0,\alpha_1}\). 
    The reason is that $D_e(W)$ usually showed up before some $D(U)$ or $D(V)$ nodes appear.  
    We are ready to attack only when we teamed up this $D_e(W)$ with either \(D(U)\) or \(D(V_{\alpha_0})\) for some \(\alpha_0\subseteq \alpha\). 
    The decision of choosing which set to attack and why it works are the core of our construction and will be discussed later. 
\end{remark*}

\begin{enumerate}[resume*]
    \item Suppose that \(\alpha\) is an \(R(\alpha_0,\alpha_1,\alpha_2)\)- or an \(R^-(\alpha_0,\alpha_1,\alpha_2)\)\nbd{}node, where \(\alpha_0\), \(\alpha_1\) and \(\alpha_2\) are \(G_{e_0}(U)\)-, \(G_{e_1}(V_{\alpha_0})\)- and \(G_{e_2}(W_{\alpha_0,\alpha_1})\)\nbd{}nodes for some \(e_0\), \(e_1\) and \(e_2\), respectively.
    Let \(W=W_{\alpha_0,\alpha_1}\) and \(V=V_{\alpha_0}\).
    \begin{enumerate}[label=(T\arabic{enumi}.\arabic*)]
        \item\label{it:pt R} If \(\alpha\) is an \(R(\alpha_0,\alpha_1,\alpha_2)\)\nbd{}node, then \(\alpha\) has infinitely many outgoing edges, labelled and ordered by
        \[
            (0,W) <_L (0,U) <_L (0,V) <_L \cdots <_L (n,W) <_L (n,U) <_L (n,V) <_L \cdots.
        \]
        Furthermore, for each \(n\in \omega\), \(\alpha\concat (n,W)\) is assigned to \(D_{e_2}(W)\), \(\alpha\concat (n,U)\) assigned to \(D_{e_0}(U)\) and \(\alpha\concat (n,V)\) assigned to \(D_{e_1}(V)\).
        \item\label{it:pt R-} If \(\alpha\) is an \(R^-(\alpha_0,\alpha_1,\alpha_2)\)\nbd{}node, then \(\alpha\) has infinitely many outgoing edges, labelled and ordered by
        \[
            (0,W) <_L (0,V) <_L \cdots <_L (n,W) <_L (n,V) <_L \cdots.
        \]
        Furthermore, for each \(n\in \omega\), \(\alpha\concat (n,W)\) is assigned to \(D_{e_2}(W)\) and \(\alpha\concat (n,V)\) assigned to \(D_{e_1}(V)\).
    \end{enumerate}
    \((n,W)\) is said to be a \emph{divergent outcome} for \(\alpha_2\); \((n,V)\) a divergent outcome for \(\alpha_1\); 
    and \((n,U)\) a divergent outcome for \(\alpha_0\) if \(\alpha\) is also an \(R(\alpha_0,\alpha_1,\alpha_2)\)\nbd{}node.
\end{enumerate}
\begin{remark*}
    A divergent outcome \((n,W)\) for \(\alpha_2\) indicates that we have some number \(m\), not necessarily \(n\) (see Section~\ref{sec:parameter} for the discussion on edge parameters), 
    such that \(\Gamma_{e_2}(BW;m)\uparrow\) and \(G_{e_2}(W)\) is therefore satisfied. 
    Extending \((n,W)\)\nbd{}outcome, we shall begin to work for the next requirement \(D_{e_2}(W)\) in the requirement list. 
    An \(R(\alpha_0,\alpha_1,\alpha_2)\)\nbd{}node is necessarily in the \(\Sigma_3\)\nbd{}world while \(R^-(\alpha_0,\alpha_1,\alpha_2)\)\nbd{}node is necessarily in a \(\Pi_3\)\nbd{}world.

\end{remark*}

\begin{enumerate}[resume*]
    \item\label{it:pt C} Suppose that \(\alpha\) is a \(C(\alpha_0,\alpha_1,\alpha_2)\)\nbd{}node, then \(\alpha\) has infinitely many outgoing edges, labelled and ordered with order type \(\omega+1\) by
    \[
        (0,U) <_L \cdots <_L (n,U) <_L \cdots <_L \omega.
    \]
    Any extension of \(\alpha\concat \omega\) is declared to be \emph{in the \(\Pi_3\)\nbd{}world of \(\alpha\)}. 
    Furthermore, \((n,U)\) is said to be a divergent outcome for \(\alpha_0\) and \(\alpha\concat (n,U)\) is assigned to \(D_{e_0}(U)\), where $e_0$ comes from the index of the \(G_{e_0}(U)\)\nbd{}node \(\alpha_0\). 
    \(\alpha\concat \omega\) is assigned to \(S(\alpha_0,\alpha_1,\alpha_2)\).
\end{enumerate}

\begin{remark*}
    At a \(C(\alpha_0,\alpha_1,\alpha_2)\)\nbd{}node \(\alpha\) we try to detect a divergent point of \(\Gamma_{e_0}(BU)\). 
    If we find one, we continue to work on the next requirement \(D_{e_0}(U)\); otherwise, we shall not work for the candidate \(U\) anymore;
    we switch to work for the candidate \(V_{\alpha_0}\) (together with some sets from $\bW$) by assigning \(\alpha\concat \omega\) to \(S(\alpha_0,\alpha_1,\alpha_2)\).
    Note that by~\ref{it:pt D-CS} there will be no more \(C\)\nbd{}node extending \(\alpha\concat \omega\) and therefore there will be no nested \(\Pi_3\)\nbd{}worlds.
\end{remark*}

\begin{enumerate}[resume*]
    \item\label{it:pt S} Suppose that \(\alpha\) is an \(S(\alpha_0,\alpha_1,\alpha_2)\)\nbd{}node, where \(\alpha_1\) is a \(G_{e_1}(V_{\alpha_0})\)\nbd{}node for some \(e_1\). Let \(V=V_{\alpha_0}\). 
    Then \(\alpha\) has infinitely many outgoing edges labelled and ordered by 
    \[
        (0,V) <_L (1,V) <_L \cdots <_L (n,V) <_L \cdots.
    \]
    Furthermore, \((n,V)\) is said to be a divergent outcome for \(\alpha_1\) and for each \(n\in \omega\), \(\alpha\concat (n,V)\) is assigned to \(D_{e_1}(V)\).
\end{enumerate}
\begin{remark*}
    An \(S(\alpha_0,\alpha_1,\alpha_2)\)\nbd{}node \(\alpha\) is necessarily in the \(\Pi_3\)\nbd{}world for some \(C\)\nbd{}node \(\beta\) with \(\beta\concat \omega\subseteq \alpha\). 
    At \(\alpha\), we try to detect a divergent point of \(\Gamma_{e_1}(BV)\). If we find one, we shall work on the next requirement \(D_{e_1}(V)\). 
    If the detection fails, we shall conclude that \(A\le_T B\) toward a contradiction.
\end{remark*}

\begin{enumerate}[resume*]
    \item\label{it:pt P} Suppose that \(\alpha\) is a \(P(\eU,\eW)\)\nbd{}node. Then \(\alpha\) is a terminal node on the priority tree. We let \(E_\bU=\{\alpha^-\}\) as has been defined in~\ref{it:pt D-P}.
    For each \(\beta\subseteq \alpha\) where \(\beta\) is a \(C(\alpha_0,\alpha_1,\alpha_2)\)\nbd{}node, we enumerate \(\alpha_2\) into \(\eW\). 
    
    \item\label{it:pt Q} Suppose that \(\alpha\) is a \(Q(\eV,\eW)\)\nbd{}node. Then \(\alpha\) is a terminal node on the priority tree. We let \(E_\bV=\{\alpha^-\}\) as has been defined in~\ref{it:pt D-Q}. For each \(\beta\subseteq \alpha\) where \(\beta\) is an \(S(\alpha_0,\alpha_1,\alpha_2)\)\nbd{}node, we enumerate \(\alpha_2\) into \(\eW\). 

\end{enumerate}
\begin{remark*}
    A \(P(\eU,\eW)\)\nbd{}node \(\alpha\) will produce many potential diagonalizing witnesses for \(\beta\in E_\bU\) or \(\xi\in E_\bW\), each of whom is a \(D\)\nbd{}node. 
    And \(\alpha\) bets that \(A\) will permit some of the diagonalizing witnesses eventually. 
    If \(A\) never permits, then \(\alpha\) will demonstrate \(A\) to be recursive toward a contradiction. 
    Otherwise, one of the \(D\)\nbd{}node reaches a \(\Sigma_1\)\nbd{}outcome and is satisfied forever. Consequently, the $P$\nbd{}node is switched off true path.
    Similar discussion applies to a \(Q\)\nbd{}node.
\end{remark*}

This finishes the definition of the priority tree \(\cT\). 




\begin{lemma}\label{lem:full requirement}
    Let \(\rho\) be an infinite path along \(\cT\). Then there is a (unique) \(X\) such that \(\Permit(X)\), and for each \(e\), \(G_e(X)\) and \(D_e(X)\) are assigned to some node in \(\rho\). 
    Moreover, for this $X$, if \(\alpha\in \rho\) is assigned to \(G_e(X)\), then either
    \begin{enumerate}
        \item \(\alpha\concat 0 \in \rho\), or
        \item \(\alpha\concat \infty \in \rho\) and there is some \(\beta\in \rho\) such that \(\beta\concat (n,X) \in \rho\) where \((n,X)\) is the divergent outcome for \(\alpha\);
    \end{enumerate}
    if \(\alpha\in \rho\) is assigned to \(D_e(X)\), then either \(\alpha\concat 0\) or \(\alpha\concat d\) is in \(\rho\).
\end{lemma}
\begin{proof}
A set \(X\) wanted by the lemma is said to be \emph{good}.

Along the infinite path \(\rho\), it is clear that if \(D_e(X)\concat 0 \in \rho\), then \(G_{e+1}(X)\in \rho\); if \(G_e(X)\concat 0\in\rho\), then \(D_e(X)\in \rho\); 
if \(G_e(X)\concat \infty\in \rho\) and there is some \(\beta\in \rho\) such that \(\beta\concat (n,X) \in \rho\) where \((n,X)\) is the \emph{divergent outcome} for \(G_e(X)\), 
then \(D_e(X)\in \rho\). Also note that \(D_e(U)\concat \infty \notin \rho\) and \(D_e(V)\concat \infty \notin \rho\) as both are assigned to either a $P$\nbd{} or a $Q$\nbd{}node, which are terminal nodes, 
contradicting that \(\rho\) is infinite.

Clearly the root is assigned to \(\Permit(U)\). Suppose that \(U\) is not good, then for some \(\alpha_0\in \rho\) assigned to \(G_{e_0}(U)\) for some \(e_0\), 
we must have that \(\alpha_0\concat \infty\in \rho\) but no divergent outcome for \(\alpha_0\) can be found in \(\rho\). 
Note that \(\alpha_0\concat\infty\) is assigned to \(\Permit(V_{\alpha_0})\). Suppose that \(V_{\alpha_0}\) is not good, 
then for some \(\alpha_1\in \rho\) assigned to \(G_{e_1}(V_{\alpha_0})\) for some \(e_1\) we must have that \(\alpha_1\concat\infty \in \rho\) but no divergent outcome for \(\alpha_1\) can be found in \(\rho\). 
Note that \(\alpha_1\concat\infty\) is assigned to \(\Permit(W_{\alpha_0,\alpha_1})\). We claim that \(W=W_{\alpha_0,\alpha_1}\) is good.

Case 1. Suppose \(D_{e}(W)\concat \infty \in \rho\). Then by~\ref{it:pt D-CS}, \(D_e(W)\concat \infty\) is either an \(S\)\nbd{}node or a \(C\)\nbd{}node. 
However, by~\ref{it:pt C}, \ref{it:pt S}, and the assumption that no divergent outcome for either \(\alpha_0\) or \(\alpha_1\) can be found in \(\rho\), we have a contradiction. 
Therefore if \(D_e(W)\in \rho\), then \(D_e(W)\concat 0\) or \(D_e(W)\concat d\) is in \(\rho\).

Case 2. Suppose \(G_e(W)\concat \infty \in \rho\). Then by~\ref{it:pt GW}, \(\xi=G_e(W)\concat \infty\) is either an \(R(\alpha_0,\alpha_1,\alpha_2)\)\nbd{}node or an \(R^-(\alpha_0,\alpha_1,\alpha_2)\)\nbd{}node. 
We should have \(\xi\concat (n,W)\in \rho\). That is, if \(G_e(W)\concat \infty\in \rho\), then \(D_{e+1}(W)\in \rho\).

Hence \(W\) is good.
\end{proof}

\begin{remark*}
We did not formally associate explicit ``requirements'' to the type of nodes labelled $R, R^-,S, C, P$ and $Q$.
If we do, the requirement associated with  \(R(\alpha_0,\alpha_1,\alpha_2)\)\nbd{}node would be: 
\[
        (\Gamma_{e_0}(BU)=A\land \Gamma_{e_1}(BV)=A \land \Gamma_{e_2}(BW)=A) \Rightarrow A=\Phi(B),
\]
similarly for $R^-$ and $S$-nodes.  
The requirements associated with $P(E_U,E_W)$-node would be:
\[
[\Delta_{e(\alpha)}=U\wedge \bigwedge_{\sigma\in E_W} \Delta_{*}=W_{\_,\_}]
\Rightarrow \Theta=A,
\]
where $\alpha$ is the unique $D_e(U)$-node in $E_U$ and $e(\alpha)$ is the index $e$; and the missing parameters in $\Delta_{*}=W_{\_,\_}$ can be found precisely because each $\sigma\in E_W$ is necessarily a
$D_{i}(W_{\alpha',\alpha''})$-node, similarly for $Q(E_V, E_W)$-nodes.
The $C(\alpha_0,\alpha_1,\alpha_2)$-nodes can be better viewed as to detect if $\Gamma_{e_0}(BU)$ is total where $e_0$ is the index of the Gamma functional at $\alpha_0$, which is part of strategy to satisfy $G_{e_0}(U)$. 
In this sense, a $C$-node plays the analog role of a $G(X)$-node on the tree which detects if there are infinitely many expansionary stages. 
From these examples, it is clear that the formalism would often involve several sets and functionals whose indices can only be found in a rather complicated way, which would not provide any further insight.

On the other hand, without spilling out the requirements associated with $R, R^-,S, C, P, Q$-nodes, we can still verify the correctness of the construction as follows: First establish the existence of an (infinite) true path $\rho$; then obtain the set \(X\) given by Lemma~\ref{lem:full requirement} from \(\rho\); finally argue that requirements \(\Permit(X)\), $G_e(X)$ and $D_e(X)$, as stated in subsection 2.2, are all satisfied for this $X$.  In this way, the requirements associated with $R, R^-,S, C, P, Q$-nodes are involved only indirectly.
\end{remark*}

From now on, we will not pay attention to the index \(e\) in a \(G_e(X)\)\nbd{}node or a \(D_e(X)\)\nbd{}node: if \(\alpha\) is a \(G_e(X)\)\nbd{}node, we will write \(\Gamma_\alpha\) for \(\Gamma_e\) and \(\gamma_\alpha\) for \(\gamma_e\); 
if \(\alpha\) is a \(D_e(X)\)\nbd{}node, we will write \(\Delta_\alpha\) for \(\Delta_e\). As there will be no ambiguities, \(\xi\concat (n,U)\) will be abbreviated as \(\xi\concat n\) if \(\xi\) is a \(C\)\nbd{}node. We leave the notation of \((n,V)\)\nbd{}outcome as it is.



\begin{definition}
    If \(\xi\) is a \(C(\alpha_0,\alpha_1,\alpha_2)\)\nbd{}node, we write 
    \[
    \cT[\xi, \Pi_3] = \{\beta\in \cT\mid \xi\concat \omega\subseteq \beta\}
    \] 
    and
    \[
        \cT[\xi, \Sigma_3] = \{\beta\in \cT\mid (\exists n)\xi\concat n\subseteq \beta\}.
    \]
\end{definition}
A node \(\beta\) is in a \(\Pi_3\)\nbd{}world if and only if there exists some \(C\)\nbd{}node \(\xi\) such that \(\beta\in \cT[\xi,\Pi_3]\).

We let \(\alpha^-\) denote the predecessor of \(\alpha\), and we write \(\type(\alpha)=X\) if \(\alpha\) is a \(\Permit(X)\)\nbd{}node, a \(G(X)\)\nbd{}node or a \(D(X)\)\nbd{}node.

\begin{definition}
    Suppose that \(\type(\alpha)=X\). If \(X=U\), then \(\per(\alpha)\) is the root of \(\cT\); 
    if \(X=V_{\alpha_0}\) for some \(\alpha_0\), then \(\per(\alpha)=\alpha_0\concat \infty\); 
    if \(X=W_{\alpha_0,\alpha_1}\) for some \(\alpha_0,\alpha_1\), then \(\per(\alpha)=\alpha_1\concat \infty\).
\end{definition}

In other words, if \(\type(\alpha)=X\), then \(\per(\alpha)\) is its ``host'' which is the \(\Permit(X)\)\nbd{}node where we introduced the set $X$.

\begin{definition} \label{def:conflict}
    Let \(\alpha\) be a \(D(W)\)\nbd{}node for some \(W=W_{\alpha_0,\alpha_1}\). 
    An \(R(\alpha_0,\alpha_1,\alpha_2)\)\nbd{}node \(\beta\) with \(\beta\concat (n,W)\subseteq \alpha\) for some \(n\) is said to have \emph{conflicts with} \(\alpha\) 
    (or equivalently, \(\alpha\) has conflicts with \(\beta\)).

    Let \(\cf(\alpha)\) collects each \(\beta\subseteq \alpha\) that has conflicts with \(\alpha\).
\end{definition}

\begin{remark*} The relation between \(\alpha\) and some $\beta\in \cf(\alpha)$ can be loosely stated as follows.  We would like to diagonalize $\Delta$ at $\alpha$ with some witness $w$ targeting $W$.  
But we will see that this $w$ must be paired with some other witness $u$ or $v$ targeting $U$ or $V$ respectively.  Whether $w$ is chosen (discarded) or one of $u,v$ is chosen (discarded) will depend on what happens at those $\beta\in \cf(\alpha)$.  Note that \(\cf(\alpha)\) can be empty.  
\end{remark*}

Now we discuss some structural properties of the priority tree \(\cT\). By~\ref{it:pt C}, we have the following
\begin{lemma}\label{lem:S locate C}
    Let \(\alpha\) be an \(S(\alpha_0,\alpha_1,\alpha_2)\)\nbd{}node, then there is a unique \(C\)\nbd{}node \(\xi\) such that \(\xi\concat\omega\subseteq \alpha\). 
    Moreover, \(\xi\) is a \(C(\alpha_0,\alpha_3,\alpha_4)\)\nbd{}node for some \(\alpha_3\) and \(\alpha_4\). \qed{}
\end{lemma}

\begin{lemma}\label{lem:Q locate C}
    Let \(\alpha\) be a \(Q(\eV,\eW)\)\nbd{}node. Then \(\eW\neq \varnothing\) if and only if there exists a unique \(C(\alpha_0,\alpha_1,\alpha_2)\)\nbd{}node \(\xi\) for some \(\alpha_0\), \(\alpha_1\), 
    and \(\alpha_2\) such that \(\alpha\in \cT[\xi,\Pi_3]\).
\end{lemma}
\begin{proof}
    \((\Leftarrow)\) Note that \(\xi\concat \omega\) is an \(S(\alpha_0,\alpha_1,\alpha_2)\)\nbd{}node and hence \(\alpha_2\in \eW\). \((\Rightarrow)\) Uniqueness is clear from the definition of priority tree as we do not have nested \(\Pi_3\)\nbd{}worlds. By~\ref{it:pt Q} and \(\eW\neq \varnothing\), there is an \(S\)\nbd{}node \(\sigma\subseteq\alpha\). Then we apply Lemma~\ref{lem:S locate C} to get the \(C\)\nbd{}node.
\end{proof}

This lemma indicates where the non-uniformity comes in: those \(Q(\eV,\eW)\)\nbd{}nodes inside a \(\Pi_3\)\nbd{}world have a choice to make when comes to diagonalization, whereas the \(Q(\eV,\varnothing)\) nodes outside \(\Pi_3\)\nbd{}worlds have no choice but to enumerate elements into \(V\). 

The following two lemmas are clear from the definition of \(\cT\).
\begin{lemma}
    Let \(\alpha\) be a \(D(X)\)\nbd{}node. Then for each \(\beta\) with \(\alpha\concat \infty \subseteq \beta\), \(\type(\beta)\neq X\). \qed{}
\end{lemma}

It says that if we have not satisfied \(D(X)\) yet, we stop working on other requirements for \(X\).


\begin{lemma}
    Let \(\alpha\) be a \(P(\eU,\eW)\)\nbd{}node. For \(\beta\in \eW\), \(\beta\mapsto \type(\beta)\) is injective. \qed{}
\end{lemma}

We now postulate the global (or static) priority $\prec$ between the nodes on the priority tree. We use \(\alpha\prec \beta\) to denote that \(\alpha\) has higher global priority than \(\beta\). 
We mostly follow the left-to-right order on the priority tree. 
\begin{definition} [global priority]\label{def:global priority}
    Let \(\alpha,\beta\in \cT\) be such that there exists some \(\xi\) such that \(\xi\concat o_0 \subseteq \alpha\), \(\xi\concat o_1\subseteq \beta\) and \(o_0 <_L o_1\).
    \begin{enumerate}
        \item\label{it:prio it1} If \(\xi\) is not a \(C\)\nbd{}node, then we define \(\alpha \prec \beta\).
        \item\label{it:prio it2} If \(\xi\) is a \(C\)\nbd{}node but \(o_1\) is not \(\omega\), then we define \(\alpha \prec \beta\).
    \end{enumerate}
    If \(\alpha\prec\beta\) or \(\beta\prec\alpha\), we say \(\alpha\) and \(\beta\) are \(\prec\)\nbd{}comparable. 
\end{definition}

If \(\alpha\subsetneq \beta\), we do not think that \(\alpha\) has higher priority than \(\beta\).
Global priorities are similar to the usual priorities because they depend on the positions of the nodes on the tree. However,
if \(\xi\) is a \(C\)\nbd{}node and \(o_1\) is \(\omega\), then we do \emph{not} define \(\prec\) between \(\alpha\) and \(\beta\); a ``dynamic'' and ``local'' priority \(\prec_{\xi,s}\) between these two will be defined in Definition~\ref{def:local priority}. We call it ``dynamic'' because it depends on the pairs (which will be introduced in Subsection~\ref{sec:Q node}), which are not permanent and can be canceled. Moreover, even if \(\alpha\) is paired with \(\beta\), the priority between them can only be determined based on what happens on current stage.

The following lemma will be used in Lemma~\ref{lem:R con}.
\begin{lemma}\label{lem:comparable nodes}
    \begin{enumerate}
        \item If \(\alpha\) and \(\beta\) are \(P(\eU,\eW)\)- and \(P(\eU',\eW')\)\nbd{}nodes respectively, then they are \(\prec\)\nbd{}comparable.
        \item If \(\alpha\) and \(\beta\) are \(Q(\{\sigma\},\eW)\)- and \(Q(\{\tau\},\eW')\)\nbd{}nodes, respectively, with \(\type(\sigma) = \type(\tau)\), then they are \(\prec\)\nbd{}comparable.
    \end{enumerate}   
\end{lemma}

\begin{proof}
(1) holds as a \(P\)\nbd{}node never belongs to a \(\Pi_3\)\nbd{}world. 

For (2), we suppose toward a contradiction that there exists a \(C(\alpha_0,\alpha_1,\alpha_2)\)\nbd{}node \(\xi\) such that \(\xi\concat n\subseteq \alpha\) and \(\xi\concat \omega \subseteq \beta\). It is clear then \(\type(\tau)=V_{\alpha_0}\neq \type(\sigma)\).

\end{proof}

\section{Parameters}\label{sec:parameter}
Unlike a usual priority argument, where we might have static parameters assigned to the edges of the priority tree, our construction has dynamic parameters; each node has to update them at the end of each stage. Two kinds of the parameters are \emph{edge parameters} and \emph{pairing parameters} and will be discussed first in this section.

We assume that the readers are familiar with usual tree constructions. We now describe how a given node \(\alpha\) maintains its parameters at the end of stage~\(s\).



\medskip\noindent
\(R(\alpha_0,\alpha_1,\alpha_2)\)\nbd{}node:
Suppose \(V=V_{\alpha_0}\) and \(W=W_{\alpha_0,\alpha_1}\). 
For each \(n\) and for each \(X\in \{U,V,W\}\), the \emph{edge parameter} \(z_{(n,X)}\) is initially set to be \(n\). 
Let \((n,X)\) be the outcome such that \(\alpha\concat (n,X)\) is visited at \(s\), if any.
\begin{enumerate}
    \item If \(X=W\), we update \(z_{(n,U)}=z_{(n,V)} = s\) and \(z_{(n+j,Y)}=s+j\) for each \(j\ge 1\) and each \(Y\in \{U,V,W\}\).
    \item If \(X=U\) or \(X=V\), we do nothing.
\end{enumerate}
If \(\alpha\) is initialized at stage~\(s\), we set \(z_{(n,X)}=n\) for each \(n\) and each \(X\in \{U,V,W\}\).

\begin{remark*}
    If \(z_{(n,U)}\) is stable (i.e., \(\lim_s z_{(n,U)}[s]<\infty\)), 
    then the divergent outcome \((n,U)\) indicates that there is some \(m\le z_{(n,U)}\) such that \(\Gamma_{\alpha_0}(BU;m)\uparrow\).
    Other divergent outcomes are similar.
\end{remark*}

\medskip\noindent
\(S(\alpha_0,\alpha_1,\alpha_2)\)\nbd{}node:
Suppose \(V=V_{\alpha_0}\) and \(W=W_{\alpha_0,\alpha_1}\). For each \(n\), the \emph{edge parameter} \(y_{(n,V)}\) is initially set to be \(n\). 
Now at stage~\(s\), 
let \(n\) be such that \(\alpha\concat (n,V)\) is visited at stage~\(s\).  We update \(y_{(n+j,V)}=s+j\) for each \(j\ge 1\).
If \(\alpha\) is initialized at stage~\(s\), we set \(y_{(n,V)}=n\).

\medskip\noindent
\(C(\alpha_0,\alpha_1,\alpha_2)\)\nbd{}node: For each \(n\), the \emph{edge parameter} \(x_{(n,U)}\) is initially set to be \(n\).
Now at stage~\(s\), let \(n\) be such that \(\alpha\concat (n,U)\) is visited at stage~\(s\), we update \(x_{(n+j,U)}=s+j\) for each \(j\ge 1\). If \(\alpha\) is initialized at stage~\(s\), we set \(x_{(n,U)}=n\).

\medskip
The above covers the nodes where edge parameters are needed. The next one concerns with pairing parameters.

\medskip\noindent
\(Q(\eV,\eW)\)\nbd{}node:
If $\eW = \varnothing$, then it needs no pairing parameters. Now we assume that \(\eW\neq \varnothing\). 
By Lemma~\ref{lem:Q locate C}, we let \(\xi\) be the (unique) \(C(\alpha_0,\alpha_1,\alpha_2)\)\nbd{}node with \(\xi\concat \omega \subseteq \alpha\).
The \emph{pairing parameter} \(\tp(\alpha)\) is defined to be the least number \(q\) such that \(q > y_{(n,V)}\) for each \(S\)\nbd{}node \(\beta\) and each \(n\) with \(\xi\concat \omega\subseteq \beta\concat (n,V) \subseteq \alpha\) (we assume that \(y_{(n,V)}\) has been updated because \(\beta\subseteq \alpha\)). 

\medskip
\(P(\eU,\eW)\)\nbd{}nodes however need no pairing parameters. 
Other kinds of parameters are discussed below.

\medskip\noindent
\(D(X)\)\nbd{}node: \(d\)\nbd{}outcome of \(\alpha\) can be \emph{activated} at some stage. Once it becomes activated, it remains activated forever (even if the node is initialized).

\medskip

In our construction, an \(R\)- or \(R^-\)\nbd{}node builds a functional \(\Phi\), an \(S\)\nbd{}node builds a functional \(\Psi\), a \(P\)\nbd{}node builds a functional \(\Theta\), and a \(Q\)\nbd{}node builds a functional \(\Theta\) (the symbol is reused). These functionals are built locally and hence will be referred as \emph{local parameters}. Once a node is initialized or disarmed, its local parameters are discarded immediately. 

A \(P(\eU,\eW)\)- or \(Q(\eV,\eW)\)\nbd{}node \(\alpha\) has to maintain a function \(\dw_\alpha\) (standing for \emph{d}iagonalizing \emph{w}itness). This \(\dw_\alpha\) is also referred as a \emph{local parameter}.
A \(P\)\nbd{}node (or \(Q\)\nbd{}node, respectively) is where we attack the recursiveness of \(U\) (or \(V\), respectively) or \(W\) by enumerating a diagoalizing witness into one of them. The whole attacking process consists of four phases and \(\dw_\alpha\) is defined in the first phase:

\begin{enumerate}[fullwidth,leftmargin=*,itemindent=3em,label=(Phase\arabic*)]
    \item\label{it:ph1} Prepare the diagonalizing witnesses (Subsections~\ref{sec:P node} and~\ref{sec:Q node}).
    Without loss of generality, we assume that \(\alpha\) is a \(P(\eU,\eW)\)\nbd{}node. \(\alpha\) has to prepare many diagonalizing witnesses for each \(\sigma\in \eU\cup\eW\) hoping one of them gets permitted. To ensure that one of them get permitted, \(\alpha\) builds a functional \(\Theta\) so that \(\Theta=A\) in case none of them gets permitted. This would imply that \(A\) is recursive toward the contradiction. The diagonalizing witnesses are organised in the following way: for each \(a\in \omega\), \(\dw_\alpha(a)\) is a map from \(\eU\cup\eW\) to \(\omega\); for each \(\sigma\in \eU\cup\eW\), \(\dw_\alpha(a)(\sigma)\in \omega\) is the diagonalizing witness targeting the set \(\type(\sigma)\). We say that \(x=\dw_\alpha(a)(\sigma)\) is \emph{prepared} if \(\Delta_\sigma(x)=\type(\sigma)(x)=0\). Thus, enumerating \(x\) into \(\type(\sigma)\) leads to the \(\Sigma_1\)\nbd{}outcome of \(\sigma\) (and we activate \(\sigma\concat d\) in this case). Once \(\dw_\alpha(a)(\sigma)\) is prepared for each \(\sigma\in \eU\cup \eW\), we define \(\Theta(a)=A(a)\). 
    
    \item\label{it:ph2} Wait for the permission from \(A\) (Subsection~\ref{sec:permitting center}). 
    Suppose that \(\Theta(a)=A(a)\) has been defined and that \(a\) enters \(A\). Now \(\Theta(a)\) becomes incorrect and we have to choose (see~\ref{it:ph3}) a \(\sigma\in \eU\cup\eW\) and get ready to enumerate \(\dw_\alpha(a)(\sigma)\) into \(\type(\sigma)\). 
    
    \item\label{it:ph3} Make decision (Subsection~\ref{sec:permitting center} and~\ref{sec:R test}). 
    In our construction, there will be nodes setting up restraints (indirectly) on \(\type(\sigma)\) for certain \(\sigma\in \eU\cup \eW\), and our decision has to be made carefully. Once the decision is made, we do not enumerate the point immediately; it will be carried out only when \(\Permit(\type(\sigma))\)\nbd{}node is visited.
    
    \item\label{it:ph4} Enumerate the point (Subsection~\ref{sec:permit node}). 
    Suppose that the chosen permission node is visited and the chosen diagonalizing point \(x=\dw_\alpha(a)(\sigma)\) has not been canceled yet, we simply enumerate it into the \(\type(\sigma)\) and activate \(\sigma\concat d\) at the same moment (and \(\alpha\) will not be visited again).
\end{enumerate}

\section{Strategies and construction}\label{sec:strategies}
When a node \(\alpha\) is visited, particular actions will be taken and will be described formally in the subroutine \(\visit(\alpha)\). In each of the following subsections, it describes either \(\visit(\alpha)\) or other actions.
We will freely use the terms \emph{true outcome} and \emph{true path} in the usual sense in an informal remarks. Though we have a slightly modified definition for true outcome and true path (Definition~\ref{def:T star}), the intuition behind them remains the same.


\begin{definition}\label{def:announces progress}
For \(\alpha\in \cT\), \(\alpha\) \emph{announces progress} if one of the following happens:
\begin{enumerate}
    \item it is \emph{visited};
    \item it is \emph{initialized};
    \item it is \emph{disarmed} (if \(\alpha\) is a \(Q\)\nbd{}node)---see Subsection~\ref{sec:disarm};
    \item it \emph{receives attention} (if \(\alpha\) is a \(P\)- or \(Q\)\nbd{}node)---see Subsection~\ref{sec:permitting center};
    \item it becomes \emph{activated} (if \(\alpha=\beta\concat d\) for some \(D(X)\)\nbd{}node \(\beta\))---see Subsection~\ref{sec:permit node}.
\end{enumerate}    
\end{definition}

\begin{convention}\label{convention: init}
    Whenever \(\alpha\) announces progress, each \(\beta\) of lower global priority (Definition~\ref{def:global priority}) is \emph{initialized} tacitly. If \(\alpha\) is initialized, each \(\beta\) extending \(\alpha\) is also initialized tacitly. In both cases, we say that \(\alpha\) initializes \(\beta\), or \(\beta\) is initialized by \(\alpha\).
\end{convention}

\begin{remark*}
The second sentence in Convention~\ref{convention: init} is to be used with~\ref{it:p2} to initialize nodes below \(\xi\concat n\) as appeared there.
\end{remark*}

``Initialize'' is used in the usual sense:
when a node is initialized, all local parameters are discarded. ``Disarm'' can be thought of as an analogy of ``initialize'' and specially designed for \(Q\)\nbd{}nodes (see Subsection~\ref{sec:Q node}).


\subsection{G-node}\label{sec:G node}
Suppose that \(\alpha\) is a \(G(X)\)\nbd{}node for some \(X\) and the current stage is \(s\).

\medskip
\noindent
\(\visit(\alpha)\):
\begin{enumerate}[nosep]
    \item If \(s\) is an expansionary stage, then \(\visit(\alpha\concat \infty)\).
    \item Otherwise, \(\visit(\alpha\concat 0)\).
\end{enumerate}

\begin{remark*}
The job of \(\alpha\) is to measure the length of agreements between \(\Gamma_{\alpha}(BX)\) and \(A\).
If the true outcome is \(0\)\nbd{}outcome, we end up with \(\Gamma_{\alpha}(BX)\neq A\) and \(G(X)\) is therefore satisfied.
Note that having infinitely many expansionary stages does not guarantee \(\Gamma_{\alpha}(BX)=A\). In fact, there could be some divergent point which will be detected by either \(C\)-, \(R\)- (or \(R^-\)-) or \(S\)\nbd{}nodes, depending on whether \(X\) is of type \(\bU\), \(\bV\) or \(\bW\).

\end{remark*}

\subsection{D-node}\label{sec:D node}
Suppose that \(\alpha\) is a \(D(X)\)\nbd{}node for some \(X\) and the current stage is \(s\).

\medskip\noindent
\(\visit(\alpha)\):
\begin{enumerate}[nosep]
    \item If \(d\)\nbd{}outcome is activated (see Subsection~\ref{sec:permit node}), then \(\visit(\alpha\concat d)\).
    \item If \(s\) is an expansionary stage, then \(\visit(\alpha\concat \infty)\).
    \item Otherwise, \(\visit(\alpha\concat 0)\).
\end{enumerate}

\begin{remark*}
The main job of \(\alpha\) is to measure the length of agreements between \(\Delta_{\alpha}\) and \(X\).
As we shall see later that once \(d\)\nbd{}outcome becomes activated, it remains activated forever.
If \(d\)\nbd{}outcome is the true outcome, then we are given a point \(n\) such that \(\Delta_\alpha(n)=0\neq 1=X(n)\) and \(D(X)\) is therefore satisfied.
If \(0\)\nbd{}outcome is the true outcome, then for some point \(n\), either \(\Delta_\alpha(n)\uparrow\) or \(\Delta_\alpha(n)\downarrow\neq X(n)\), which implies that \(D(X)\) is satisfied.
If \(\infty\)\nbd{}outcome is the true outcome, then \(\Delta_\alpha=X\) and hence \(D(X)\) is not satisfied. In this case, the set \(X\) is not a good one in the sense of Lemma~\ref{lem:full requirement}.
\end{remark*}

\subsection{P-node}\label{sec:P node} 
Suppose that \(\alpha\) is a \(P(\eU,\eW)\)\nbd{}node and \(\alpha\) is being visited at stage~\(s\).
Let \(s^*\le s\) be the least \(\alpha\)\nbd{}stage such that \(\alpha\) never gets initialized between \(s^*\) and \(s\).

\medskip\noindent
\(\visit(\alpha)\): Let \(a\ge s^*\) be the least such that \(\Theta(a)\) is not defined.
\begin{enumerate}[nosep]
    \item If \(\dw_\alpha(a)\) is not defined, we pick a fresh number \(x_\beta\) for each \(\beta\in \eU\cup \eW\) and define \(\dw_\alpha(a)(\beta)=x_\beta\). Stop the current stage.
    \item If \(\dw_\alpha(a)\) is defined but for some \(\beta\in \eU\cup\eW\) we have \(\Delta_{\beta}(x_{\beta})\uparrow\) (i.e., \(x_\beta\) is not prepared), then we stop the current stage.
    \item Otherwise, we define \(\Theta(a)=A(a)[s]\). Stop the current stage.
\end{enumerate}

\begin{remark*} 
Basically, it just implements the~\ref{it:ph1} to prepare the diagonalizing witnesses.  
In (2), it means although we expect infinitely many expansionary stages, yet the length of agreements has not reached $x_{\beta}$. So we just wait. 
As \(P\)\nbd{}node is a terminal node, the true outcome does not matter. 
As we shall see later, if \(\alpha\) is on the true path, then we have \(\Theta(n)=A(n)\) for each \(n\ge s^*\); hence, \(A\) is recursive. 
To achieve this, whenever we have \(\Theta(a)=0\neq 1=A(a)\), we have to prevent \(\alpha\) from being visited again (i.e., switch it off the true path) by diagonalizing against some \(\beta\in \eU\cup\eW\) and activating \(\beta\concat d\). See Subsection~\ref{sec:permitting center} for details.
\end{remark*}

\subsection{Q-node}\label{sec:Q node} 
Suppose that \(\alpha\) is a \(Q(\eV,\eW)\)\nbd{}node and \(\alpha\) is being visited at stage~\(s\).
Let \(s^*\le s\) be the least \(\alpha\)\nbd{}stage such that \(\alpha\) never gets initialized or disarmed (see below) between \(s^*\) and \(s\).

\medskip\noindent
\(\visit(\alpha)\):
\begin{enumerate}[nosep]
    \item\label{it:Q1} Suppose \(\eW\neq \varnothing\) and that \(\alpha\) is not \emph{paired}. 
    We let \(\xi\) be the unique \(C\)\nbd{}node with \(\xi\concat \omega\subseteq \alpha\) (Lemma~\ref{lem:Q locate C}).  
    We set up \((\xi\concat n, \alpha)\) as a \emph{pair} where \(n\) is the least (which always exists by the discussion in Section~\ref{sec:parameter}) such that
    \begin{enumerate} 
        \item 
        \(x_{(n,U)}>\tp(\alpha)\), and
        \item for each \(m\ge n\), \(\xi\concat m\) is not paired.
    \end{enumerate}
    Stop the current stage.
    
    \item\label{it:Q2} Suppose \(\eW=\varnothing\) or that \(\alpha\) has been paired. 
    Let \(a\ge s^*\) be the least such that \(\Theta(a)\) is not defined.
    \begin{enumerate}
        \item Suppose that \(\dw_\alpha(a)\) is not defined, we pick a fresh number \(x_\beta\) for each \(\beta\in \eV\cup \eW\) and define \(\dw_\alpha(a)(\beta)=x_\beta\). Stop the current stage.
        \item Suppose that \(\dw_\alpha(a)\) is defined but for some \(\beta\in \eV\cup\eW\) we have \(\Delta_{\beta}(x_{\beta})\uparrow\). Stop the current stage.
        \item Otherwise, we define \(\Theta(a)=A(a)[s]\). Stop the current stage.
    \end{enumerate}
\end{enumerate}

\begin{remark*}
The action of $Q$ has two parts: In~(\ref{it:Q2}), it prepares the diagonalizing witnesses like the \(P\)\nbd{}node above; in~(\ref{it:Q1}), it establishes a pair.
$\Sigma_3$\nbd{} and $\Pi_3$\nbd{}worlds will interact via such pairs, based on which a ``local'' ``dynamic'' priority will be defined (see Subsection~\ref{sec:C node}, particularly Definition~\ref{def:local priority}). 
Examples of pairs can be found in Figure~\ref{fig:local priority 1}, in which \((C\concat n_0,Q_0)\) and \((C\concat n_1,Q_1)\) are two pairs.
\end{remark*}
\begin{figure}
    \centering
    \includegraphics[width=\textwidth]{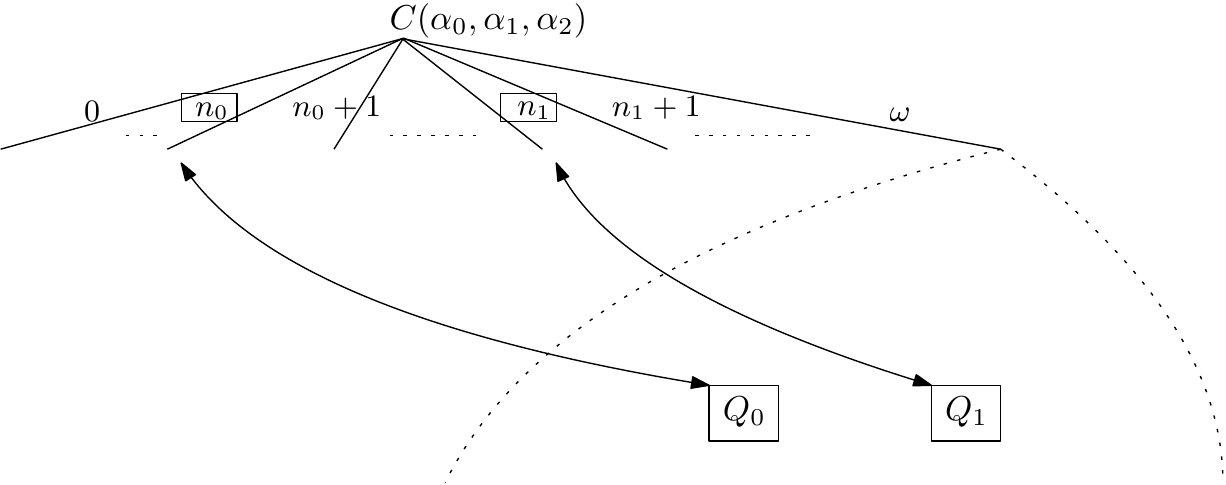}
    \caption{Examples of Pairs}\label{fig:local priority 1}
\end{figure}

\subsection{Pair}\label{sec:a pair}
Suppose that \((\xi\concat n, \alpha)\) is a pair established as in~(\ref{it:Q1}) in Subsection~\ref{sec:Q node}. We implement the following three rules:
\begin{enumerate}[label=(P\arabic*)]
    \item\label{it:p1} We \emph{cancel} the pair if and only if \(\alpha\) is initialized.
    \item\label{it:p2} Whenever the pair is canceled, we initialize $\xi\concat n$ (and all nodes below $\xi\concat n$ by Convention~\ref{convention: init}).
    \item\label{it:p3} Whenever \(\alpha\) announces progress (Definition~\ref{def:announces progress}), we initialize \(\xi\concat (n+1)\).
\end{enumerate}
Note that the pair is \emph{not} canceled when \(\alpha\) is disarmed (Subsection~\ref{sec:disarm}). 
A case of special importance shall be noted:
\begin{lemma}\label{lem:cleared of P}
    Let \((\xi\concat n,\alpha)\) be a pair. Whenever \((\alpha^-)\concat d\) is activated, \(\xi\concat n\) is initialized.
\end{lemma}

\begin{proof}
    Suppose that \((\alpha^-)\concat d\) is activated and therefore \(\alpha\) is initialized. By~\ref{it:p1}, the pair is canceled. By~\ref{it:p2}, \(\xi\concat n\) is initialized.
\end{proof}

\begin{definition}\label{def:pairing priority}
    For two pairs, we say that \((\xi\concat n,\alpha)\) has higher \emph{pairing priority} than \((\xi\concat m,\beta)\) if \(n<m\). (Note that \(n<m\) if and only if \(\alpha\prec\beta\).)
\end{definition}

\subsection{R-node}\label{sec:R node}
Suppose that \(\alpha\) is an \(R(\alpha_0,\alpha_1,\alpha_2)\)\nbd{}node or an \(R^-(\alpha_0,\alpha_1,\alpha_2)\)\nbd{}node and the current stage is \(s\). Let \(s^*<s\) be the last \(\alpha\)\nbd{}stage.

\medskip\noindent
\(\visit(\alpha)\): Let \(n=0\).
\begin{enumerate}[nosep]
    \item Suppose \(n>s\). We stop the current stage.
    \item Suppose \(\Phi(B;z_{(n,W)})\downarrow[s]\). We proceed with \(n+1\).
    \item\label{it:R 3} Suppose \(\Phi(B;z_{(n,W)})\uparrow[s]\) but \(\Phi(B;z_{(n,W)})\downarrow[s^*]\).
    \begin{enumerate}
        \item If \(\gamma_{\alpha_2}(z_{(n,W)})[s]>\gamma_{\alpha_2}(z_{(n,W)})[s^*]\), \(\visit(\alpha\concat (n,W))\).
        \item (Ignored this line if \(\alpha\) is an \(R^-\)\nbd{}node) If (a) does not happen and \(\gamma_{\alpha_0}(z_{(n,U)})[s]>\gamma_{\alpha_0}(z_{(n,U)})[s^*]\), \(\visit(\alpha\concat (n,U))\).
        \item If neither (a) nor (b) happens and \(\gamma_{\alpha_1}(z_{(n,V)})[s]>\gamma_{\alpha_1}(z_{(n,V)})[s^*]\), \(\visit(\alpha\concat (n,V))\).
        \item Otherwise, we define \(\Phi(B;z)=A(z)[s]\)  with use \(\varphi(B;z_{(n,W)})[s^*]\) for each \(z\) with \(z_{(n-1,W)}<z\le z_{(n,W)}\) and then stop the current stage.
    \end{enumerate}
    \item Otherwise, for each \(z\) with \(z_{(n-1,W)}<z\le z_{(n,W)}\), we define \(\Phi(B;z)[s]=A(z)[s]\) with a fresh use \(u\), particularly
    \begin{enumerate}
        \item\label{it:R 4a} \(u\ge \gamma_{\alpha_2}(BW;z_{(n,W)})[s]\),
        \item\label{it:R 4b} (Ignore this line if \(\alpha\) is an \(R^-\)\nbd{}node) \(u\ge \gamma_{\alpha_0}(BU;z_{(n,U)})[s]\), and
        \item\label{it:R 4c} \(u\ge \gamma_{\alpha_1}(BV;z_{(n,V)})[s]\).
    \end{enumerate}
    Then we proceed to \(n+1\).
\end{enumerate}

\begin{remark*}
The $R$\nbd{}node is responsible for defining $\Phi(B; z_{(n,W)})=A(z_{(n,W)})$, whose use $\varphi(z_{(n,W)})$ is set to bound  $\gamma_{\alpha_2}(BW;z_{(n,W)})[s], \gamma_{\alpha_0}(BU;z_{(n,U)})[s]$ and 
$\gamma_{\alpha_1}(BV;z_{(n,V)})[s]$.  When $\Phi(B;z_{(n,W)})$ becomes undefined (as in~(\ref{it:R 3}) above), we track the culprit and
try to demonstrate (in that order) that $\Gamma_{\alpha_2}(BW;z_{(n,W)})$, $\Gamma_{\alpha_0}(BU;z_{(n,U)})$ or
$\Gamma_{\alpha_1}(BV;z_{(n,V)})$ is undefined. 
It is important to note that (\ref{it:R 4b}) and~(\ref{it:R 4c}) are ensured only at the stage when we define \(\Phi(B;z_{(n,W)})\); the two conditions are allowed to be violated at other stages.
The reader shall find details in Subsection~\ref{sec:R strategy}. 
\end{remark*}

\subsection{S-node}\label{sec:S node}
Suppose that \(\alpha\) is an \(S(\alpha_0,\alpha_1,\alpha_2)\)\nbd{}node and the current stage is \(s\). Let \(s^*<s\) be the last \(\alpha\)\nbd{}stage.

\medskip\noindent
\(\visit(\alpha)\): Let \(n=0\).
\begin{enumerate}[nosep]
    \item Suppose \(n>s\). We stop the current stage.
    \item Suppose \(\Psi(B;y_{(n,V)})\downarrow[s]\). We
    proceed with \(n+1\).
    \item Suppose \(\Psi(B;y_{(n,V)})\uparrow[s]\) but \(\Psi(B;y_{(n,V)})\downarrow[s^*]\).
    \begin{enumerate}
        \item If \(\gamma_{\alpha_1}(BV;y_{(n,V)})[s]>\gamma_{\alpha_1}(BV;y_{(n,V)})[s^*]\), then \(\visit(\alpha\concat (n,V))\).
        \item Otherwise, we define \(\Psi(B;y)=A(y)[s]\) with use \(\psi(B;y_{(n,V)})[s^*]\) for each \(y\) with \(y_{(n-1,V)}<y\le y_{(n,V)}\).
        Then we stop the current stage.
    \end{enumerate}
    \item Otherwise, for each \(y\) with \(y_{(n-1,V)}<y\le y_{(n,V)}\), we define \(\Psi(B;y)=A(y)[s]\) with a fresh use \(u\), particularly \(u >  \gamma_{\alpha_1}(BV;y_{(n,V)})[s]\).
    Then we
    proceed to \(n+1\).
\end{enumerate}

\begin{remark*}
The actions of $S$\nbd{}nodes are similar to those of $R$\nbd{}nodes.  We define $\Psi(B;y)=A(y)$ with use $\psi(y)$
bounding $\gamma_{\alpha_1}(BV;y)$. If \((n,V)\) is the true outcome, we can conclude that \(\Psi(B;y_{(n,V)})\uparrow\), which implies \(\Gamma_{\alpha_1}(BV;y_{(n,V)})\uparrow\) and a divergent point for \(\Gamma_{\alpha_1}(BV)\) is found. 
Unexpectedly we might end up with \(\Gamma_{\alpha_1}(BV;y_{(n,V)})\uparrow\) but \(\Psi(B;y_{(n,V)})\downarrow\). This may happen when \(\Gamma_{\alpha_0}(BU;y_{(n,U)})\downarrow\) with \(\gamma_{\alpha_0}(BU;y_{(n,U)})<\psi(B;y_{(n,V)})\) whereas \(\visit(\alpha)\) does not say anything about \(\Gamma_{\alpha_0}(BU)\). The success of an \(S\)\nbd{}node relies on the ``totality'' (in a less strict sense) of \(\Gamma_{\alpha_0}(BU)\). Such reliance is the core of our construction and will be discussed in Subsection~\ref{sec:S strategy}.
\end{remark*}

\subsection{C-node and local dynamic priorities}\label{sec:C node}
Let \(\xi\) be a \(C(\alpha_0,\alpha_1,\alpha_2)\)\nbd{}node.
Before we describe \(\visit(\xi)\) at stage~\(s\), we need to define the ``local'' ``dynamic'' priority $\prec_{\xi,s}$ which connects some pairs of nodes that one is in the $\Sigma_3$- and the other is in the $\Pi_3$\nbd{}worlds of $\xi$. 
It contrasts the global priority \(\prec\) (without subscripts) defined in Definition~\ref{def:global priority} which depends only on the priority tree. In the notation \(\prec_{\xi,s}\), the subscript \(\xi\) corresponds to the ``local'' part and \(s\) corresponds to the ``dynamic'' part.

Recall that the procedure of establishing a pair was introduced in (\ref{it:Q1}) in Subsection~\ref{sec:Q node}. 

\begin{definition}[local priority]\label{def:local priority}
At stage~\(s\), let 
\[
    (\xi\concat n_0,\alpha_0), \ldots, (\xi\concat n_{k-1}, \alpha_{k-1}).
\]
be the list of all pairs in descending order of pairing priority (Definition~\ref{def:pairing priority}).

For \(\eta\in \cT[\xi,\Sigma_3]\) and \(\beta\in \cT[\xi,\Pi_3]\) (Note that $\eta$ and $\beta$ are not $\prec$\nbd{}comparable), we define \(\prec_{\xi,s}\) as follows:
\begin{enumerate}
    \item\label{it:local priority 1} \(\eta\prec_{\xi,s} \beta\) if there exists some \(i<k\) such that \(\eta\) extends \(\xi\concat m\) for some \(m\le n_i\) and \(\alpha_i\preceq \beta\),
    \item\label{it:local priority 2} \(\beta\prec_{\xi,s} \eta\) if there exists some \(i<k\) such that \(\beta\preceq \alpha_i\) and \(\eta\) extends \(\xi\concat m\) for some \(n_i<m\). 
\end{enumerate}

(If the list is empty, then \(\prec_{\xi,s}=\varnothing\).)
\end{definition}

\begin{remark*}
We will often write \(\prec_\xi\) for short as the stage~\(s\) is usually clear from context. 
Note also that \(\prec_\xi\) is a partial order; \(\eta\in \cT[\xi,\Sigma_3]\) and \(\beta\in \cT[\xi,\Pi_3]\) can be \(\prec_\xi\)\nbd{}incomparable. 
However, if \(\eta\) is a \(P\)- or \(Q\)\nbd{}node and \(\beta\) is a \(Q\)\nbd{}node that is paired, then either \(\eta\prec_\xi\beta\) or \(\beta\prec_\xi\eta\). 

In Figure~\ref{fig:local priority 2}, two pairs \((C\concat n_0,Q_0)\) and \((C\concat n_1,Q_1)\) are given. An arrow \(\alpha\to\beta\) indicates \(\alpha\prec\beta\) if \(\alpha\) and \(\beta\) lie in the same row; \(\alpha\prec_\xi \beta\) if they lie in different rows: labels along the arrows refer to corresponding items in Definition~\ref{def:local priority}. For example, Item~(\ref{it:local priority 1}) of Definition~\ref{def:local priority} gives, without an explicit arrow in Figure~\ref{fig:local priority 2}, \(C\concat 0\prec_\xi \beta_1\), which follows equivalently from \(C\concat 0\prec C\concat n_0\prec_\xi Q_0\prec \beta_1\). In a similar fashion, Item~(\ref{it:local priority 2}) of Definition~\ref{def:local priority} gives \(\beta_1\prec_\xi C\concat(n_1+1)\), and we might view this as a consequence of \(\beta_1\prec Q_1\prec_\xi C\concat(n_1+1)\). Note that \(C\concat 0\) and \(\beta_0\) are not \(\prec_\xi\)\nbd{}comparable. Although \(\beta_0\) and \(C\concat n_0\) are not \(\prec_\xi\)\nbd{}comparable, we observe that whenever \(\beta_0\) announces progress, \(Q_0\) is initialized, and by~\ref{it:p1} (Subsection~\ref{sec:a pair}) the pair \((C\concat n_0,Q_0)\) is canceled and by~\ref{it:p2} \(C\concat n_0\) is therefore initialized (Lemma~\ref{lem:cleared of P} is a special case of this phenomenon). Also,~\ref{it:p3} is justified by Item~(\ref{it:local priority 2}) of Definition~\ref{def:local priority}.
\end{remark*}

\begin{figure}
    \centering
    \includegraphics[width=\textwidth]{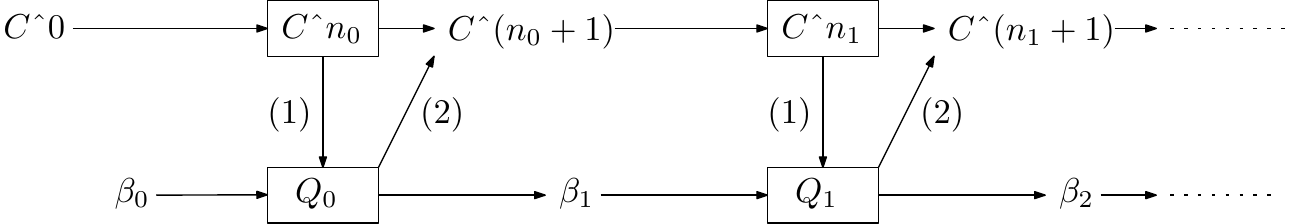}
    \caption{An example of Local Dynamic Priority}\label{fig:local priority 2}
\end{figure}

With the definition of local dynamic priority, we are ready to describe \(\visit(\xi)\) for the \(C(\alpha_0,\alpha_1,\alpha_2)\)\nbd{}node \(\xi\).

\medskip\noindent
\(\visit(\xi)\):
Let \(s^*<s\) be the last \(\xi\)\nbd{}stage. Let \(n\) be the least (if any) such that \(\gamma_{\alpha_0}(BU;x_{(n,U)})[s]>\gamma_{\alpha_0}(BU;x_{(n,U)})[s^*]\). We are ready to visit nodes in \(\cT[\xi,\Pi_3]\). We start with \(\eta=\xi\concat \omega\). Whenever we run into \(\visit(\eta)\) for \(\eta\in \cT[\xi,\Pi_3]\), we interrupt the subroutine and do the following:
\begin{enumerate}
    \item If \(\abs{\eta}>s\) and \(\xi\concat n\) has not been initialized, \(\visit(\xi\concat n)\); if \(\abs{\eta}>s\) and \(\xi\concat n\) has been initialized, we stop the current stage. \emph{(After nodes in \(\cT[\xi,\Pi_3]\) has been visited, we decide if we need to visit nodes in \(\cT[\xi,\Sigma_3]\).)}
    
    \item If \(\xi\concat n\prec_\xi \eta\), then \(\visit(\xi\concat n)\). \emph{(We stop visiting nodes in \(\cT[\xi,\Pi_3]\).)}
    
    
    \item Otherwise, \(\visit(\eta)\). (\(\xi\concat n\) could potentially be initialized by~\ref{it:p1}, \ref{it:p2} or~\ref{it:p3} when certain \(\eta\) is visited.)
\end{enumerate}
If \(n\) does not exist, \(\visit(\xi\concat \omega)\) without interruptions.



\begin{remark*} At the $C$\nbd{}node, we are facing two parallel worlds which are connected via pairs.
Instead of having two separated accessible path, we explore the two possibilities simultaneously.
Whenever we see an interaction via a pair, we stay in the winner's world. If we do not see any interaction, we do not bet as either world can be the winner; indeed, we visit both worlds simultaneously.
\end{remark*}

\subsection{Permitting center}\label{sec:permitting center}
The permitting center is not represented on the tree. At the beginning of a stage~\(s\), if we have \(A(a)[s-1]\neq A(a)[s]\) for some \(a\), then some \(P\)- or \(Q\)\nbd{}node \(\beta\) might have a problem because \(\Theta(a)\neq A(a)[s]\). 
\(\Theta(a)\) cannot be corrected and we have to prevent \(\beta\) from being visited again by diagonalizing against some ``good'' (whose existence is one of the most important parts of the verification) \(\alpha\in \eU\cup\eW\) (if \(\beta\) is a \(P(\eU,\eW)\)\nbd{}node) or \(\alpha\in \eV\cup\eW\) (if \(\beta\) is a \(Q(\eV,\eW)\)\nbd{}node) and activating \(\alpha\concat d\), which has higher global priority than \(\beta\). This section is to discuss how we choose this \(\alpha\) and this is done at the permitting center at the beginning of each stage.

\begin{remark*}
All tensions among nodes on the priority tree trace back to a \(\Theta\) defined by either a \(P\)- or \(Q\)\nbd{}node and a point \(a\) such that \(\Theta(a)\neq A(a)\).
\end{remark*}

Let \(\lambda\) be the root of the priority tree \(\cT\) and \(s\) be the current stage.

\medskip\noindent
Permitting center:
\begin{enumerate}
    \item\label{it:pc 1} If there is no point \(a\) with \(A(a)[s-1]\neq A(a)[s]\), we do nothing.
    \item\label{it:pc 2} (making choice) Let \(a\) be such that \(A(a)[s-1]\neq A(a)[s]\). Let \(\beta\) be the \(P(\eU,\eW)\)- or \(Q(\eV,\eW)\)\nbd{}node, if any, which has not received attention and \(\Theta(a)=0\). We list \(\eW\) as \(\xi_0\subseteq \dots \subseteq \xi_{k-1}\) where \(k=\abs{\eW}\) (the list is empty if \(\eW=\varnothing\)). We define the choice function \(\chi\) by letting 
    \[
        \chi(\beta;a)=\xi_i    
    \]
    for the least \(i\), if any, such that for each \(\alpha\in \cf(\xi_i)\) we have (see Subsection~\ref{sec:R test} for the definition of \(\Test\) function)
    \[
        \Test(\alpha; a, \dw_\beta(a)(\xi_i))=1;
    \]
    if such \(i\) does not exist (including the case when \(\eW=\varnothing\)),
    \[
        \chi(\beta;a)=\beta^-.
    \]
    \item\label{it:pc 3} Let \(\sigma=\chi(\beta;a)\). We \emph{send} the point \(\dw_\beta(a)(\sigma)\) to \(\per(\sigma)\), which is the \(\Permit(X)\)\nbd{}node where \(X=\type(\sigma)\).
\end{enumerate}
If~(\ref{it:pc 2}) happens, we say that \(\beta\) \emph{receives attention} (on \(a\) at stage~\(s\)).

\begin{remark*}
\begin{enumerate}[label=(\roman*)]
\item For convenience, \(\beta\) is allowed to receive attention only once since last initialization or disarmament.
\item There can exist at most one \(\beta\) as desired in~(\ref{it:pc 2}). To see this, let \(\beta_0\) and \(\beta_1\) be two \(P\)- or \(Q\)\nbd{}nodes who build \(\Theta_0\) and \(\Theta_1\) respectively. 
If both of them are \(P\)\nbd{}nodes (or \(Q\)\nbd{}nodes), they are \(\prec\)\nbd{}comparable; otherwise, being terminal nodes, there must exist some \(\xi\) such that \(\beta_0\in \cT[\xi,\Sigma_3]\) and \(\beta_1\in \cT[\xi,\Pi_3]\), say. By the remark below Definition~\ref{def:local priority}, they are \(\prec_\xi\)\nbd{}comparable.
Therefore, whenever we visit the node of higher global (or local) priority, the one of lower global (or local) priority is initialized or disarmed.
Therefore, domains of~\(\Theta_0\) and \(\Theta_1\) must be disjoint.
\item \(\chi(\beta;a)\) chooses the node~\(\sigma\) such that we are allowed to enumerate the point \(\dw_\beta(a)(\sigma)\) into~\(X\). However, we are not enumerating it into~\(X\) immediately; we will do so when we visit \(\per(\sigma)\) provided that \(\beta\) has not been disarmed or initialized. The delay feature of the permitting is crucial to Lemma~\ref{lem:R correct}.
\end{enumerate}
\end{remark*}

\subsection{R-test}\label{sec:R test}
Intuitively, an \(R\)-node would like to put a restraint on \(W\) while \(\Phi(B;z_{(n,W)})\downarrow\). 
If a \(P(\eU,\eW)\)\nbd{}node \(\beta\), for example, receives attention on \(a\), we are not free to enumerate the point \(w=\dw_\beta(a)(\sigma)\) into \(\type(\sigma)\) (where \(\sigma\in\eW\)): an \(R\)\nbd{}node \(\alpha\in \cf(\sigma)\) may still have a restraint on \(W\). 
To tell in which situation an \(R\)\nbd{}node \(\alpha\in \cf(\sigma)\) allows \(w\) to be enumerated in a delayed fashion, we now introduce the \(\Test\) function, which we have used without definition in Step~(\ref{it:pc 2}) of permitting center.
The following definition is a bit technical and the whole picture can only be seen in the verification section, particularly Subsection~\ref{sec:R strategy}.   
\begin{definition}[R-test]\label{def:R test}
    Let \(\alpha\) be an \(R(\alpha_0,\alpha_1,\alpha_2)\)\nbd{}node. Fix numbers \(a\) and \(w=\dw_\beta(a)(\sigma)\) for some \(\beta\) and \(\sigma\). \(\beta\) is assumed to be either a \(P(\eU,\eW)\)- or a \(Q(\eV,\eW)\)\nbd{}node with \(\sigma\in \eW\) and \(\alpha\in \cf(\sigma)\). 

    Suppose that \(\beta\) is a \(P(\eU,\eW)\)\nbd{}node. For every $n$, we define \(\test(\alpha,n;a,w)[s]{=1}\) if one of the following holds
    \begin{enumerate}[label=(R-t\arabic*)]
        \item\label{it:R-t1} \(\Phi(B;z_{(n,W)})\uparrow[s]\);
        \item\label{it:R-t2} \(\Phi(B;z_{(n,W)})\downarrow[s]\) and \(w>\varphi(B;z_{(n,W)})[s]\); or
        \item\label{it:R-t3} \(\Phi(B;z_{(n,W)})\downarrow[s]\), 
        \(\bigsame(U;\gamma_{\alpha_0}(BU;a)[s^*],s^*,s)\) and \(\bigsame(B;\gamma_{\alpha_0}(BU;a)[s^*],s^*,s)\),
        where \(s^*\le s\) is the last stage when we define \(\Phi(B;z_{(n,W)})\).
    \end{enumerate}
    
    Suppose that \(\beta\) is a \(Q(\eV,\eW)\)\nbd{}node. We define \(\test(\alpha,n;a,w)[s]=1\) if either~\ref{it:R-t1}, \ref{it:R-t2} or the following holds
    \begin{enumerate}[resume*]
        \item\label{it:R-t4} \(\Phi(B;z_{(n,W)})\downarrow[s]\),
        \(\bigsame(V;\gamma_{\alpha_1}(BV;a)[s^*],s^*,s)\) and \(\bigsame(B;\gamma_{\alpha_1}(BV;a)[s^*],s^*,s)\),
        where \(s^*\le s\) is the last stage when we define \(\Phi(B;z_{(n,W)})\).
    \end{enumerate}

    Define \(\Test(\alpha;a,w)[s]=1\) if for every \(n\) we have \(\test(\alpha,n;a,w)[s]=1\).
\end{definition}


\begin{remark*}
Informally, $\test(\alpha,n;a,w)[s]=1$ means that $w$ passed the test posed by $\alpha$ so that the computation $\Phi(B;z_{(n,W)})$ would not be affected by $w$ entering $W$.
Items~\ref{it:R-t1} and~\ref{it:R-t2} are natural, whereas items~\ref{it:R-t3} and~\ref{it:R-t4} are toward that the computation $\Phi(B; z_{(n,W)})$ is now ``protected'' by the $U$- and $V$\nbd{}side respectively.


We could define $\test(\alpha,n;a,w)[s]=0$ if $\test(\alpha,n;a,w)[s]\neq 1$, but since we never use it, we leave it as it is.
\end{remark*}

If \(\test(\alpha,n;a,w)=1\) via~\ref{it:R-t2}, then \(\test(\alpha,m;a,w)=1\) for each \(m<n\); 
if \(\test(\alpha,n;a,w)=1\) via~\ref{it:R-t1}, then \(\test(\alpha,m;a,w)=1\) for each \(m>n\). In fact, we have the following lemma to simplify the test, where the number \(n\) in the lemma is also easy to determine.

\begin{lemma}\label{lem:test}
    At stage~\(s\), let \[
        n = \max \{ m \mid \test(\alpha, m; a, w) = 1 \text{ via~\ref{it:R-t2}}\} + 1.
    \]
    Suppose $\test(\alpha, n; a, w) = 1$, then $\Test(\alpha; a, w) = 1$. 
\end{lemma}


\begin{proof}
    Since \(\test(\alpha,n;a,w)=1\) via~\ref{it:R-t1} implies that \(\Test(\alpha;a,w)=1\), we assume, without loss of generality, that \(\alpha\) is a \(P\)\nbd{}node and \(\test(\alpha,n;a,w)=1\) via~\ref{it:R-t3}. 
    Hence we have \(\bigsame(U;\gamma_{\alpha_0}(BU;a)[s^*],s^*,s)\) and \(\bigsame(B;\gamma_{\alpha_0}(BU;a)[s^*],s^*,s)\), where \(s^*\le s\) is the last stage when we define \(\Phi(B;z_{(n,W)})\).
    Let \(m>n\) be such that \(\Phi(B;z_{(m,W)})\downarrow[s]\) and \(s^{**}\le s\) be the last stage when we define \(\Phi(B;z_{(m,W)})\). As \(s^{*}\le s^{**}\), therefore we have \(\bigsame(U;\gamma_{\alpha_0}(BU;a)[s^{**}],s^{**},s)\) and \(\bigsame(B;\gamma_{\alpha_0}(BU;a)[s^{**}],s^{**},s)\). Hence \(\test(\alpha,m;a,w)=1\). 
\end{proof}

\subsection{Permit(X)-node}\label{sec:permit node}
Suppose \(\alpha\) is a \(\Permit(X)\)\nbd{}node and the current stage is \(s\). Recall that at Step~(\ref{it:pc 3}) of permitting center, some point was sent to a \(\Permit(X)\)\nbd{}node. 

\medskip\noindent
\(\visit(\alpha)\): Let \(x=\dw_\beta(a)(\sigma)\) be the point sent by permitting center~(\ref{it:pc 3}), where \(\sigma=\chi(\beta;a)\), \(\per(\sigma)=\alpha\) and \(\type(\sigma)=X\). (We assume that \(x\) is not discarded yet).
\begin{enumerate}
    \item\label{it:Pi3 init Sigma3 2} We enumerate \(x\) into \(X\) and activate the \(d\)\nbd{}outcome of \(\sigma\).
    \item Then \(\visit(\alpha\concat 0)\).
\end{enumerate}

\begin{remark*}
    Besides Convention~\ref{convention: init}, Item (1) above is also accompanied by an initialization as pointed out in Lemma~\ref{lem:cleared of P}.
\end{remark*}

\subsection{Disarming process}\label{sec:disarm}

Let \(s\) be the current stage.
Let \(\beta\) be a \(Q(\eV,\eW)\)\nbd{}node with \(\eW\neq \varnothing\), \(\alpha\) be an \(S(\alpha_0,\alpha_1,\alpha_2)\)\nbd{}node, 
\(\xi\) be the unique \(C(\alpha_0,\alpha_3,\alpha_4)\)\nbd{}node with \(\xi\concat \omega \subseteq \alpha \subsetneq \alpha\concat n\subseteq \beta\), and \((\xi\concat m,\beta)\) is paired at some \(s_0 < s\).

\medskip\noindent
\((\xi\concat m, \alpha\concat n,\beta)\)-DisarmingProcess: We disarm \(\beta\) immediately if one of the following events is observed at stage~\(s\):
\begin{enumerate}[label=(DP-\arabic*)]
    \item\label{it:DP-1} At the permitting center, some \(P\)- or \(Q\)\nbd{}node extending $\xi\concat l$ with $l \leq m$ receives attention.
    
    
    \item\label{it:DP-2} At \(\Permit(U)\)\nbd{}node, both \(\lnot\,\bigsame(U;\gamma_{\alpha_0}(BU;x_{(m,U)})[s_1],s_1,s)\) and \\
    \(\bigsame(B;\psi_\alpha(B;y_{(n,V)})[s_1],s_1,s)\) hold, where \(s_1<s\) is the last \(\alpha\)\nbd{}stage when we define \(\Psi_\alpha(B;y_{(n,V)})\). 
    
    \item\label{it:DP-3} At the node \(\xi\), the outcome to be visited is \(\xi\concat l\) for some \(l\le m\).
    
    
    \item\label{it:DP-4} At the node \(\alpha\),  \(\psi_\alpha(B;y_{(n,V)})\downarrow[s]<\gamma_{\alpha_0}(BU;x_{(m,U)})[s]\) and for the last \(\alpha\)\nbd{}stage~\(s_1<s\) we have either \(\psi_\alpha(B;y_{(n,V)})\uparrow[s_1]\) or \(\psi_\alpha(B;y_{(n,V)})\downarrow[s_1]\ge \gamma_{\alpha_0}(BU;x_{(m,U)})[s_1]\).
    
    \item\label{it:DP-5} At the end of the stage when we maintain the global parameters, \(x_{(m,U)}[s-1]\neq x_{(m,U)}[s]\).
\end{enumerate}
If one of the events is observed, we say that \((\xi\concat m, \alpha\concat n,\beta)\)-DisarmingProcess is \emph{triggered}.
The \((\xi\concat m, \alpha\concat n,\beta)\)-DisarmingProcess lasts until \((\xi\concat m,\beta)\) is canceled.

When \(\beta\) is disarmed, we discard its \(\Theta\) and \(\dw\).
If there is a diagonalizing witness \(x=\dw_\beta(a)(\sigma)\) for some \(a\) and \(\sigma\in \eV\cup\eW\) having been sent to \(\per(\sigma)\) (as described in permitting center~(\ref{it:pc 3})), we also discard the diagonalizing witness \(x\).

\begin{remark*}
\ref{it:DP-5} follows from~\ref{it:DP-1} and~\ref{it:DP-3} and we state it explicitly for convenience. In fact, \ref{it:DP-2} needs a bit of justification. Since~\ref{it:DP-5} implies \(x_{(m,U)}[s_1]=x_{(m,U)}[s]\), \(x_{(m,U)}\) in~\ref{it:DP-2} is legitimate. 
There could be multiple disarming processes but they are working independently respecting Convention~\ref{convention: init}.
The latter part of~\ref{it:DP-4} is there only to prevent one from disarming $\beta$ infinitely often if $\psi(B; y_{(n,V)})\downarrow < \gamma_{\alpha_0}(BU;x_{(m,U)})$.   

\end{remark*}
\subsection{Construction}\label{sec:construction}
At the beginning of stage~\(s\), we let the permitting center acts and then \(\visit(\lambda)\) where \(\lambda\) is the root of the priority tree \(\cT\). Meanwhile, the disarming processes are working in the background. We stop the current stage whenever \(\visit(\alpha)\) with \(\abs{\alpha}=s\).

This is the end of the construction.

\section{Verification}\label{sec:con and ver}
A \emph{path} of a tree is a collection of nodes such that they are pairwise comparable and are closed under initial segments. The standard way to pick up the \emph{``true path''} is to collect inductively those leftmost nodes that are visited infinitely often. In our construction, however, such nodes can be initialized infinitely often. Therefore we shall slightly modify the definition of the true path.

\begin{definition}\label{def:T star}
    Let \(\cT\) be the priority tree and the construction be given as above. We define 
    \[
        \cT^*=\{\alpha\in \cT\mid \text{\(\alpha\) is visited infinitely often and \(\alpha\) is \emph{injured} finitely often}\},
    \]
    where ``injured'' means ``initialized or disarmed''.
    The leftmost \emph{infinite} path \(\rho\) of \(\cT^*\), if any, is the \emph{true path}. For \(\alpha\in\rho\), \(o\)\nbd{}outcome is the \emph{true outcome} if \(\alpha\concat o\in \rho\).
\end{definition}

Clearly \(\lambda\in \cT^*\), \(\cT^*\) is a two-branching (the labels are not recursive) tree by Convention~\ref{convention: init}, and \(\cT^*\) is \(0''\)\nbd{}recursive. In Lemma~\ref{lem:T star path} we show that \(\cT^*\) has an infinite path. In fact, \(\cT^*\) can have multiple infinite paths, each of which is in fact as good as the true path --- we just pick one in order to discuss the index of a winning set (Lemma~\ref{lem:index}).

In Subsections~\ref{sec:R strategy} and~\ref{sec:S strategy}, we show that \(R\)- and \(S\)\nbd{}nodes can actually work as intended at each stage. Therefore nodes in \(\cT^*\) achieve their goals (Lemma~\ref{lem:T star}). In Subsection~\ref{sec:proof of main}, we argue that \(\cT^*\) has an infinite path \(\rho\) and all requirements for the set \(X\) (given by Lemma~\ref{lem:full requirement}) are satisfied.







\subsection{\texorpdfstring{Success of $R$-strategies}{success of R-strategies}}\label{sec:R strategy}

We will not distinguish between an \(R\)\nbd{}node and an \(R^{-}\)\nbd{}node in this section: the following definition also applies to an \(R^-\)\nbd{}node. Lemma~\ref{lem:R con} is strictly easier to verify (more precisely, Case~1 does not happen).

\begin{definition} [R-Condition]\label{def:R con}
Let \(\alpha\) be an \(R(\alpha_0,\alpha_1,\alpha_2)\)\nbd{}node. We say that \(R(\alpha,n)\)\nbd{}Condition holds at stage~\(t\) if either
\begin{enumerate}
    \item \(\Phi(B;z_{(n,W)})\uparrow[t]\), or
    \item \(\Phi(B;z_{(n,W)})\downarrow[t]\) but \(\bigsame(W,\varphi(B;z_{(n,W)})[s],s,t)\) where \(s\le t\) is the last stage when we define \(\varphi(B;z_{(n,W)})\).
\end{enumerate}
We say that \(R(\alpha)\)\nbd{}Condition holds if \(R(\alpha,n)\)\nbd{}Condition holds for each \(n\).
\end{definition}

Whenever we define \(\Phi(B;z_{(n,W)})\) at a stage~\(s\), \(R(\alpha,n)\) holds at stage~\(s\) trivially.

The purpose of \(R\)\nbd{}Condition is to ensure the correctness of $\Phi$ defined at $R$:

\begin{lemma}[R-Correctness]\label{lem:R correct}
     Let \(\alpha\) be an \(R(\alpha_0,\alpha_1,\alpha_2)\)\nbd{}node, and assume that $R(\alpha)$\nbd{}Condition holds at every stage. For each \(\alpha\)\nbd{}stage~\(s\) and each \(n\le s\), if \(\Phi(B;z_{(n,W)})\downarrow[s]\), then \(\Phi(B;z_{(n,W)})\downarrow[s] = A(z_{(n,W)})[s]\).
\end{lemma}

\begin{proof}
    Fix an $\alpha$\nbd{}stage~\(s\) and a number $n$ such that \(\Phi(B;z_{(n,W)})\)[s] is defined. 
    Let \(s_0<s\) be the last $\alpha$\nbd{}stage when we defined \(\Phi(B;z_{(n,W)})\). 
    Therefore we have
    \begin{enumerate}
        \item \(\Phi(B;z_{(n,W)})[s_0]=A(z_{(n,W)})[s_0]\),
        \item \(\bigsame(B;\varphi(B;z_{(n,W)})[s_0],s_0,s)\), and
        \item \(\varphi(B;z_{(n,W)})[s]=\varphi(B;z_{(n,W)})[s_0]>\gamma_{\alpha_2}(BW;z_{(n,W)})[s_0]\).
    \end{enumerate}
    Furthermore, since \(R(\alpha,n)\)\nbd{}Condition holds, we also have
    \begin{enumerate}[resume]
        \item \(\bigsame(W;\varphi(B;z_{(n,W)})[s_0],s_0,s)\).
    \end{enumerate}

    Hence we have
    \begin{align*}
    \Phi(B;z_{(n,W)})[s] &= \Phi(B;z_{(n,W)})[s_0]\ \ \ (\mbox{by (2)})\\
    &= A(z_{(n,W)})[s_0]\ \ \ (\mbox{by (1)})\\
    &= \Gamma_{\alpha_2}(BW;z_{(n,W)})[s_0]\ \  (\mbox{since $s_0$ is $\alpha_2$\nbd{}expansionary})\\
    &= \Gamma_{\alpha_2}(BW;z_{(n,W)})[s] \ \ \ (\mbox{by (2), (3) and (4)})\\
    &= A(z_{(n,W)})[s] \ \ \ (\mbox{since $s$ is $\alpha_2$\nbd{}expansionary}). \qedhere
    \end{align*}
\end{proof}

By Lemma~\ref{lem:R correct}, it suffices to check that our construction ensures \(R\)\nbd{}condition.

Note that if \(\test(\alpha,n;a,w)[s]=1\) via~\ref{it:R-t1} or~\ref{it:R-t2}, enumerating \(w\) into \(W_{\alpha_0,\alpha_1}\) does not violate \(R(\alpha,n)\)\nbd{}condition; if \(\test(\alpha,n;a,w)[s]=1\) via~(R-t3), we will send \(w\) to \(\Permit(W_{\alpha_0,\alpha_1})\) and will enumerate \(w\) into \(W\) whenever \(\Permit(W_{\alpha_0,\alpha_1})\) is visited unless \(w\) is discarded.


\begin{lemma} [R-Condition Satisfaction]\label{lem:R con}
    For each \(R\)\nbd{}node \(\alpha\) and each stage~\(t\), \(R(\alpha)\)\nbd{}Condition holds.
\end{lemma}

\begin{proof}
Let \(\alpha\) be an \(R(\alpha_0,\alpha_1,\alpha_2)\)\nbd{}node, \(V=V_{\alpha_0}\) and \(W=W_{\alpha_0,\alpha_1}\). Let \(s < t\) be the last \(\alpha\)\nbd{}stage when we defined \(\varphi(B;z_{(n,W)})\). Suppose that we have \(\Phi(B;z_{(n,W)})\downarrow [t]\) but \(\lnot \bigsame(W,\varphi(B;z_{(n,W)})[s],s,t)\). Then we have
\begin{enumerate}
    \item \(\bigsame(B;\varphi(B;z_{(n,W)})[s];s,t)\),
    \item \(\varphi(B;z_{(n,W)})[s]>\gamma_{\alpha_1}(BV;z_{(n,V)})[s]\),
    \item \(\varphi(B;z_{(n,W)})[s]>\gamma_{\alpha_0}(BU;z_{(n,U)})[s]\), and 
    \item for some \(w\le \varphi(B;z_{(n,W)})[s]\), \(w\) is enumerated into \(W\), say, at some (least) stage~\(s_1\) with \(s<s_1\le t\).
\end{enumerate}

We assume that at some stage~\(s_2\) with \(s<s_2\le s_1\), \(\beta\) receives attention due to~\ref{it:ph3} of the attacking process for some \(\beta\). 

\textbf{Case 1.} Suppose that \(\beta\) is a \(P(\eU,\eW)\)\nbd{}node (therefore \(\alpha\) is not an \(R^-\)\nbd{}node) and \(\chi(\beta;a)=\xi\in \eW\) for some \(a\) such that \(\alpha\in \cf(\xi)\), \(\dw_\beta(a)(\xi)=w\) and \(\type(\xi)=W\). From (4) we conclude that \(\per(\xi)\) is visited at \(s_1\).

As \(w\le \varphi(B;z_{(n,W)})[s]\), we conclude that \(\beta\) extends \(\alpha\concat (m, W) \) for some \(m\le n\) and therefore 
\begin{enumerate}[resume]
    \item \(a<z_{(n,U)}\) (by the setting of Section~\ref{sec:parameter}).
\end{enumerate}
Since \(\Test(\alpha;a,w)=1\) at \(s_2\) and hence \(\test(\alpha,n;a,w)=1\), we must have~\ref{it:R-t3}, which implies
\begin{enumerate}[resume]
    \item \(\bigsame(U;\gamma_{\alpha_0}(BU;a)[s],s,s_2)\).
\end{enumerate}

We observe that only \(P\)\nbd{}node can potentially enumerate a point into \(U\), and any two \(P\)\nbd{}nodes are always \(\prec\)\nbd{}comparable by Lemma~\ref{lem:comparable nodes}.
Also, no \(P\)\nbd{}node of higher global priority than \(\beta\) receives attention between \(s_2\) and \(s_1\),
because otherwise \(\beta\) would be initialized before we enumerate \(w\) into \(W\) at \(s_1\). 
Furthermore, any \(P\)\nbd{}node of lower global priority than \(\beta\) is initialized at \(s_2\). 
By (1), $\alpha\concat (n, W)$ is never visited between $s_2$ and $s_1$. Therefore, $w'$, prepared by any lower global priority $P$\nbd{}node between $s_2$ and $s_1$, must be larger than $\phi(B;z_{(n,W)})$, hence, $w' > \gamma_{\alpha_0}(BU;a)[s]$.
Therefore we have \(\bigsame(U;\gamma_{\alpha_0}(BU;a),s_2,s_1)\). Together with (6), we have
\begin{enumerate}[resume]
    \item \(\bigsame(U;\gamma_{\alpha_0}(BU;a),s,s_1)\).
\end{enumerate}




Note that \(\alpha_0\concat \infty \subseteq \Permit(W_{\alpha_0,\alpha_1})=\per(\xi)\). We have the following contradiction:
\begin{eqnarray*}
     0 = A(a)[s] &=& \Gamma_{\alpha_0}(BU;a)[s]\ \ \ (\mbox{since $s$ is $\alpha_0$\nbd{}expansionary})\\
    &=& \Gamma_{\alpha_0}(BU;a)[s_1]\ \ \ (\mbox{by (1),(3),(5) and (7)})\\
    &=& A(a)[s_1] = 1 \ \ \ (\mbox{since $s_1$ is $\alpha_0$\nbd{}expansionary}).\\
\end{eqnarray*}

\textbf{Case 2.} (This case varies only slightly.) Suppose that \(\beta\) is a \(Q(\eV,\eW)\)\nbd{}node and \(\chi(\beta;a)=\xi\in \eW\) for some \(a\) such that \(\alpha\in \cf(\xi)\), \(\dw_\beta(a)(\xi)=w\) and \(\type(\xi)=W\). For the same reason as in Case~1 we have 
\begin{enumerate}[resume]
    \item \(a<z_{(n,V)}\).
\end{enumerate}

Since \(\Test(\alpha;a,w)=1\) at \(s_2\) and hence \(\test(\alpha,n;a,w)=1\), we must have~\ref{it:R-t4}, which implies
\begin{enumerate}[resume]
    \item \(\bigsame(V;\gamma_{\alpha_1}(BV;a)[s],s,s_2)\).
\end{enumerate}

To show \(\bigsame(V;\gamma_{\alpha_1}(BV;a)[s],s_2,s_1)\), we first realize that only a \(Q(\eV',\eW')\)\nbd{}node \(\beta'\) with \(\eV'=\{\tau\}\) and \(\type(\tau)=V_{\alpha_0}\) potentially enumerates a point into the set \(V=V_{\alpha_0}\). By Lemma~\ref{lem:comparable nodes}, \(\beta\) and \(\beta'\) are \(\prec\)\nbd{}comparable. Such \(Q\)\nbd{}nodes of lower global priority than \(\beta\) are initialized at \(s_2\) and those of higher priorities, if received attention before or at the beginning of stage~\(s_1\), would initialize \(\beta\). Therefore we have \(\bigsame(V;\gamma_{\alpha_1}(BV;a)[s],s_2,s_1)\). Together with ($9$), we conclude that
\begin{enumerate}[resume]
    \item \(\bigsame(V;\gamma_{\alpha_1}(BV;a)[s],s,s_1)\)
\end{enumerate}

Note that \(\alpha_1\concat \infty\subseteq \Permit(W_{\alpha_0,\alpha_1})=\per(\xi)\). We have the following contradiction:
\begin{align*}
0 = A(a)[s] &= \Gamma_{\alpha_1}(BV;a)[s]\ \ \ (\mbox{since $s$ is $\alpha_1$\nbd{}expansionary})\\
&= \Gamma_{\alpha_1}(BV;a)[s_1]\ \ \ (\mbox{by (1), (2), (8) and (10)})\\
&= A(a)[s_1]=1\ \ \ (\mbox{since $s_1$ is $\alpha_1$\nbd{}expansionary}).  \qedhere    
\end{align*}
\end{proof}

\subsection{\texorpdfstring{Success of $S$-strategies}{success of S-strategies}}\label{sec:S strategy}
\begin{definition} [S-Condition]\label{def:S con}
    Let \(\alpha\) be an \(S(\alpha_0,\alpha_1,\alpha_2)\)\nbd{}node. For each $n\in \omega$, we say that \(S(\alpha,n)\)\nbd{}Condition holds at stage~\(t\) if either
    \begin{enumerate}
        \item \(\Psi(B;y_{(n,V)})\uparrow[t]\), or
        \item \(\Psi(B;y_{(n,V)})\downarrow[t]\) but we also have the following: (Here, \(s\le t\) is the last stage~when we defined \(\Psi(B;y_{(n,V)})\).)
        \begin{enumerate}
        \item\label{it:S con V} If \(\psi(B;y_{(n,V)})[s]<\gamma_{\alpha_0}(BU;y_{(n,V)})[s]\), then\\ 
        \(\bigsame(V;\psi(B;y_{(n,V)})[s],s,t)\).
        \item\label{it:S con UV} If \(\psi(B;y_{(n,V)})[s] \geq \gamma_{\alpha_0}(BU;y_{(n,V)})[s]\), then either \\
        \(\bigsame(V;\psi(B;y_{(n,V)})[s],s,t)\) or \(\bigsame(U;\gamma_{\alpha_0}(BU;y_{(n,V)})[s],s,t)\).
        \end{enumerate}
    \end{enumerate}
    We say that \(S(\alpha)\)\nbd{}Condition holds if \(S(\alpha,n)\)\nbd{}Condition holds for each \(n\).
\end{definition}


\begin{remark*}
The objective of the S-condition is to guarantee the correctness of the value of $\Psi(B; y_{(n,V)})$.
This is done by ensuring $\psi(B;y)$ is bigger than either $\gamma_{\alpha_0}(BU;y_{(n,V)})[s]$ or
$\gamma_{\alpha_1}(BV;y_{(n,V)})[s]$. 
In the case of Definition~\ref{def:S con}(\ref{it:S con V}), we have no choice but to prevent \(V\) from changing, whereas in the case of~(\ref{it:S con UV}) we are more flexible and we can ``preserve'' either $U$\nbd{}side or $V$\nbd{}side, a situation similar to the minimal pair construction. 
\end{remark*}

By a similar argument as in Lemma~\ref{lem:R correct}, one can prove the following:
\begin{lemma} [S-Correctness]\label{lem:S correct}
    Let \(\alpha\) be an \(S(\alpha_0,\alpha_1,\alpha_2)\)\nbd{}node. Assume that $S(\alpha)$\nbd{}Condition holds at every stage.  
    Then for each \(\alpha\)\nbd{}stage and each \(n\le s\), if \(\Psi(B;y_{(n,V)})\downarrow[s]\), then \(\Psi(B;y_{(n,V)})\downarrow[s]=A(x)[s]\).
\end{lemma}

\begin{proof}
    Fix an $\alpha$\nbd{}stage~\(s\) and a number $n$ such that \(\Psi(B;y_{(n,V)})[s]\) is defined. 
    Let \(s_0<s\) be the last $\alpha$\nbd{}stage when we defined \(\Psi(B;y_{(n,V)})\). 
    Therefore we have
    \begin{enumerate}
        \item \(\Psi(B;y_{(n,V)})[s_0]=A(y_{(n,V)})[s_0]\),
        \item \(\bigsame(B;\psi(B;y_{(n,V)})[s_0],s_0,s)\), and
        \item \(\psi(B;y_{(n,V)})[s]=\psi(B;y_{(n,V)})[s_0]>\gamma_{\alpha_1}(BV;y_{(n,V)})[s_0]\).
    \end{enumerate}
    By the assumption that \(S(\alpha,n)\)\nbd{}Condition holds, either Definition~\ref{def:S con}(\ref{it:S con V}) or (\ref{it:S con UV}) holds. 
    Let us first prove the former case, where we have
    \begin{enumerate}[resume]
        \item \(\bigsame(V;\psi(B;y_{(n,V)})[s_0],s_0,s)\).
    \end{enumerate}
    Hence we have
    \begin{eqnarray*}
     \Psi(B;y_{(n,V)})[s]& = & \Psi(B;y_{(n,V)})[s_0] \ \ \ (\mbox{by (2)})\\
     &=& A(y_{(n,V)})[s_0] \ \ \ (\mbox{by (1)})\\
     &=& \Gamma_{\alpha_1}(BV;y_{(n,V)})[s_0]\ \ (\mbox{since \(s_0\) is \(\alpha_1\)\nbd{}expansionary})\\
     &=& \Gamma_{\alpha_1}(BV;y_{(n,V)})[s]\ \ \ (\mbox{by (2), (3) and (4)})\\
     &=& A(y_{(n,V)})[s] \ \ \ (\mbox{since \(s\) is \(\alpha_1\)\nbd{}expansionary}).
    \end{eqnarray*}
    
    For the latter case Definition~\ref{def:S con}(\ref{it:S con UV}), we have
    
    \begin{enumerate}[resume]
        \item \(\bigsame(V;\psi(B;y_{(n,V)})[s_0],s_0,s)\), or 
        \item \(\bigsame(U;\gamma_{\alpha_0}(BU;y_{(n,V)})[s_0],s_0,s)\).
    \end{enumerate}
    
    In the case of (5), the proof is exactly the same as above. In the case of (6), we have 
    \begin{align*}
    \Psi(B;y_{(n,V)})[s]& =\Psi(B;y_{(n,V)})[s_0] \ \ \ (\mbox{by (2)})\\
    &=A(y_{(n,V)})[s_0] \ \ \ (\mbox{by (1)})\\
    &=\Gamma_{\alpha_0}(BU;y_{(n,V)})[s_0]\ \ (\mbox{since \(s_0\) is \(\alpha_0\)\nbd{}expansionary})\\
    &=\Gamma_{\alpha_0}(BU;y_{(n,V)})[s]\ \ \ (\mbox{by (2), (6) and 
    Definition~\ref{def:S con}(\ref{it:S con UV})})\\
    &=A(y_{(n,V)})[s] \ \ \ (\mbox{since \(s\) is \(\alpha_0\)\nbd{}expansionary}). \qedhere    
    \end{align*}
\end{proof}

\begin{lemma} [S-Condition Satisfaction]\label{lem:S con}
    For each \(S\)\nbd{}node \(\alpha\) and each stage~\(t\), \(S(\alpha)\)\nbd{}Condition holds.
\end{lemma}

\begin{proof}
Fixing an \(n\), we show that \(S(\alpha,n)\)\nbd{}Condition holds at stage~\(t\). 

    Suppose that \(\alpha\) is \(S(\alpha_0,\alpha_1,\alpha_2)\)\nbd{}node, where \(\alpha_0\) is the \(G(U)\)\nbd{}node, \(\alpha_1\) is the \(G(V_{\alpha_0})\)\nbd{}node and \(\alpha_2\) is the \(D(W_{\alpha_0,\alpha_1})\)\nbd{}node. We write \(V=V_{\alpha_0}\) and \(W=W_{\alpha_0,\alpha_1}\). Let \(s\le t\) be the last \(\alpha\)\nbd{}stage when we defined \(\Psi(B;y_{(n,V)})\) (and \(\alpha\) has not been initialized between \(s\) and \(t\)). Since \(S(\alpha,n)\)\nbd{}Condition trivially holds if \(\Psi(B;y_{(n,V)})\uparrow[t]\), we assume \(\Psi(B;y_{(n,V)})\downarrow[t]\). We have
    \begin{enumerate}
        \item\label{it:s1} \(\psi(B;y_{(n,V)})[s]>\gamma_{\alpha_1}(BV;y_{(n,V)})[s]\) and
        \item\label{it:s2} \(\bigsame(B;\psi(B;y_{(n,V)})[s],s,t)\).
    \end{enumerate}
    As a consequence of~(\ref{it:s2}), we do not visit \(\alpha\concat (m,V)\) for \(m\le n\) between \(s\) and \(t\).
    By Lemma~\ref{lem:S locate C}, there exists the unique \(C(\alpha_0,\alpha_3,\alpha_4)\)\nbd{}node \(\xi\) for some \(\alpha_3\) and \(\alpha_4\) with \(\xi\concat \omega\subseteq \alpha\). We list, if any, all paired \(Q\)\nbd{}nodes extending \(\alpha\concat (n, V)\) as
    \[
        \beta_0\prec \beta_1 \prec\cdots\prec\beta_{k-1}.
    \]
    Let \((\xi\concat n_i,\beta_i)\) be the established pair for each \(i<k\). 
     By the definition of pairing parameter (Section~\ref{sec:parameter} and Subsection~\ref{sec:Q node}), we have
    \begin{enumerate}[resume]
        \item\label{it:s3} \(y_{(n,V)} \le \tp(\beta_0) < x_{(n_0,U)}[s]<x_{(n_1,U)}[s]<\cdots < x_{(n_{k-1},U)}[s]\).
    \end{enumerate}
    Corresponding to clauses~(\ref{it:S con V}) and~(\ref{it:S con UV}) in \(S(\alpha,n)\)\nbd{}Condition (see Definition~\ref{def:S con}), our proof splits into two cases.

    \textbf{Case 1.} Suppose \(\psi(B;y_{(n,V)})[s]<\gamma_{\alpha_0}(BU;y_{(n,V)})[s]\).
    We show \(\bigsame(V;\psi(B;y_{(n,V)})[s],s,t)\). Let \(s_1\) be the least stage between \(s\) and \(t\) at which some number \(a\) enters \(A\) and some \(Q\)\nbd{}node \(\beta\) receives attention. 
    The goal is to show that \(\beta\) will not enumerate any point less than \(\psi(B;y_{(n,V)})[s]\) into \(V\). 
    First of all, by the choice of \(s_1\), we have 
    \begin{enumerate}[resume]
        \item\label{it:s4} \(\bigsame(V;\psi(B;y_{(n,V)})[s],s,s_1)\).
    \end{enumerate}
    We may assume that \(\beta\) extends \(\alpha\concat (m,V)\) for some \(m\) as other cases are simpler. Suppose that \(w=\dw_\beta(a)(\alpha_2)\) and \(v=\dw_\beta(a)(\beta^-)<\psi(B;y_{(n,V)})[s]\) are two diagonalizing witnesses of \(\beta\). 
    \begin{claim}\label{cl:1}
        \(m<n\).
    \end{claim}
    \begin{proof}[Proof of Claim~\ref{cl:1}]
        If \( m > n\), the diagonalizing witnesses (particularly, \(v\)) of \(\beta\) is larger than \(\psi(B;y_{(n, V)})[s]\), which contradicts our assumption.
        
        Now we show \(m\neq n\).
        We have \(\psi(B;y_{(n,V)})[s] <\gamma_{\alpha_0}(BU;y_{(n,V)})[s] <\gamma_{\alpha_0}(BU;x_{(n_0,U)})[s]\) by the assumption of Case~1 and~(\ref{it:s3}). By \((\xi\concat n_0,\alpha\concat n,\beta_0)\)\nbd{}DisarmingProcess~\ref{it:DP-4}, \(\beta_0\) has been disarmed by $s$ (and \(\beta_i\) for \(i>0\) are also initialized). 
        As we do not visit $\alpha\concat (n,V)$ between $s$ and $t$, we have \(m \neq n\).
        
    \end{proof}
    By Claim~\ref{cl:1} and the definition of \(y_{(n,V)}\) (see Section~\ref{sec:parameter}), we have \(a<y_{(n,V)}\). Therefore we have \(\psi(B;y_{(n,V)})[s]>\gamma_{\alpha_1}(BV;y_{(n,V)})[s]>\gamma_{\alpha_1}(BV;a)[s]\) by~(\ref{it:s1}). By~(\ref{it:s2}) and~(\ref{it:s4}) respectively, we have 
    \begin{enumerate}[resume]
        \item\label{it:s5} \(\bigsame(B;\gamma_{\alpha_1}(BV;a)[s],s,s_1)\), and
        \item\label{it:s6} \(\bigsame(V;\gamma_{\alpha_1}(BV;a)[s],s,s_1)\).
    \end{enumerate}
    
    \begin{claim}\label{cl:2}
        For each \(\eta\in \cf(\alpha_2)\), \(\Test(\eta;a,w)=1\).
    \end{claim}
    \begin{proof}[Proof of Claim~\ref{cl:2}]
        Let \(l\) be such that \(\eta\concat (l,W)\subseteq\alpha_2\). It suffices to show \(\test(\eta,l;a,w)=1\) by Lemma~\ref{lem:test}.
        As it is an $\alpha$\nbd{}stage, $s$ is also an $\eta\concat (l, W)$\nbd{}stage, which implies \(\Phi_\eta(B;z_{(l,W)})\uparrow[s]\).
        If \(\Phi_\eta(B;z_{(l,W)})\uparrow[s_1]\), then $\test(\eta,l;a,w)[s_1] = 1$ by Definition~\ref{def:R test}~\ref{it:R-t1}. Otherwise, there exists some stage~$s^*$ such that $s < s^* < s_1$ and $s^*$ is the last stage that \(\Phi_\eta(B;z_{(l,W)})\downarrow[s^*]\). With~(\ref{it:s5}) and~(\ref{it:s6}), we conclude \(\test(\eta,l;a,w)[s_1]=1\) by Definition~\ref{def:R test} (R-t4).    
    \end{proof}
    By Claim~\ref{cl:2} we have \(\bD(\beta,a)[s_1]\subseteq \alpha_2\). In particular, \(\bD(\beta,a)[s_1]\neq \beta^-\) and therefore \(\beta\) will not enumerate \(v\) into \(V\). Hence we have \(\bigsame(V;\psi(B;y_{(n,V)})[s],s,t)\), which completes the proof of Case 1.
    

     
    \textbf{Case 2.} Suppose \(\psi(B;y_{(n,V)})[s]\ge \gamma_{\alpha_0}(BU;y_{(n,V)})[s]\). 
    Depending on whether \(U\) or \(V\) changes first, we have two subcases.
    
    \textbf{Case 2(a).} \(U\) changes first. We are to show \(\bigsame(V;\psi(B;y_{(n,V)})[s],s,t)\). First of all, let~\(s_1\) be the least stage between \(s\) and \(t\) such that \(\bigsame(V;\psi(B;y_{(n,V)}),s,s_1)\) and \(\lnot\,\bigsame(U;\gamma_{\alpha_0}(BU,y_{(n,V)})[s],s,s_1)\). As we have \(y_{(n,V)} < x_{(n_0, U)}\) by~(\ref{it:s3}), we have \(\lnot\,\bigsame(U;\gamma_{\alpha_0}(BU,x_{(n_0, U)})[s],s,s_1)\), which triggers \((\xi\concat n_0, \alpha\concat n,\beta_0)\)\nbd{}DisarmingProcess~\ref{it:DP-2} at \(s_1\) and therefore \(\beta_0\) is disarmed (and \(\beta_i\) for \(i>0\) are initialized).
    Let \(\beta\) be a \(Q\)\nbd{}node receiving attention (say, \(a\) enters \(A\)) between \(s\) and \(t\), extending \(\alpha\concat (m,V)\) for some \(m\). By an argument similar to Claim~\ref{cl:1} and Claim~\ref{cl:2}, we conclude \(\bD(\beta,a)\subseteq\alpha_2\). Hence  \(\bigsame(V;\psi(B;y_{(n,V)})[s],s,t)\).
    
    
    
    \textbf{Case 2(b).} \(V\) changes first. We are to show \(\bigsame(U;\gamma_{\alpha_0}(BU;y_{(n,V)})[s],s,t)\). Let \(s_1\) be the least stage between \(s\) and \(t\) such that \(\bigsame(U;\gamma_{\alpha_0}(BU;y_{(n,V)})[s],s,s_1)\) and \(\lnot\,\bigsame(V;\psi(B;y_{(n,V)}),s,s_1)\). We first examine what happens at \(s_1\). Let \(\beta\) be the \(Q\)\nbd{}node who is responsible for this change. That is, \(\beta\) receives attention (say, \(a\) enters \(A\)) at \(s_2\le s_1\) with \(\bD(\beta,a)=\beta^-\) and therefore enumerates its diagonalizing witness \(v=\dw_\beta(a)(\beta^-)<\psi(B;y_{(n,v)})\) into \(V\) at \(s_1\).
    
    \begin{claim}\label{cl:3}
    \(\beta\) extends \(\alpha \concat (n,V)\).
    \end{claim}
    \begin{proof}[Proof of Claim~\ref{cl:3}]
    Suppose that \(\beta\) extends \(\alpha\concat (m,V)\) for some \(m\). If \(m>n\), then we have \(v>\psi(B;y_{(n,v)})\), contradicting the assumption. If \(m<n\), by a similar argument in Claim~\ref{cl:2}, we conclude \(\bD(\beta,a)\subseteq\alpha_2\), contradicting \(\bD(\beta,a)=\beta^-\).
    \end{proof}
    
    By Claim~\ref{cl:3}, we have \(\beta=\beta_j\) for some \(j<k\). Therefore, we have
    \begin{enumerate}[resume]
        \item\label{it:s7} \(\gamma_{\alpha_0}(BU;x_{(n_j, U)})[s] \leq \psi(B;y_{(n,V)}) [s]\) (as otherwise, the \((\xi\concat n_j, \alpha\concat n,\beta_j)\)\nbd{}DisarmingProcess~\ref{it:DP-4} would have disarmed \(\beta_j\) by \(s\));
        \item\label{it:s7b} \(\bigsame(B;\gamma_{ \alpha_0}(BU; x_{(n_j,U)})[s], s, t)\) (by~(\ref{it:s2}) and~(\ref{it:s7})); and
        \item\label{it:s8} all P-nodes below \(\xi\concat n_j\) are initialized at \(s_1\) by Lemma~\ref{lem:cleared of P}.
    \end{enumerate}
    
    Let \(s_3\) be the least stage between \(s\) and \(t\) such that \(b\) enters \(A\) and some P\nbd{}node \(\iota\) receives attention with diagonalizing witnesses \(w=\dw_\iota(b)(\alpha_4)\) and \(u=\dw_\iota(b)(\iota^-)<\gamma_{\alpha_0}(BU;x_{(n_j,U)})\). 
    Suppose that \(\iota\) extends \(\xi\concat l\) for some \(l\). 
    We have 
    \begin{claim}\label{cl:4}
        \(l<n_j\) and \(s_3>s_1\).
    \end{claim}
    \begin{proof}[Proof of Claim~\ref{cl:4}]
        If \(l> n_j\), then we have \(u>\gamma_{\alpha_0}(BU;x_{(n_j,U)})\) contradicting the assumption. If \(l\le n_j\) and \(s_3\le s_1\), then \(\beta=\beta_j\) would be disarmed by \((\xi\concat n_j, \alpha\concat n,\beta_j)\)\nbd{}DisarmingProcess~\ref{it:DP-1} at \(s_3\). If \(l=n_j\) and \(s_3>s_1\), we have a contradiction as \(\iota\) would have been initialized at \(s_1\) by~(\ref{it:s8}) (and~(\ref{it:s7b}) prevents us from visiting \(\iota\) after~\(s_1\)). Therefore, we have \(l<n_j\) and \(s_3>s_1\). 
    \end{proof}
    By Claim~\ref{cl:4}, we have \(b<x_{(n_j,U)}\) and \(\gamma_{\alpha_0}(BU;b)[s]<\gamma_{\alpha_0}(BU;x_{(n_j,U)})[s]\). Now we have
    \begin{enumerate}[resume]
        \item\label{it:s9} \(\bigsame(B; \gamma_{\alpha_0}(BU; b)[s], s, s_3)\) (by~(\ref{it:s7b})), and
        \item\label{it:s10b} \(\bigsame(U; \gamma_{\alpha_0}(BU; b)[s], s, s_3)\) (by the choice of \(s_3\)).
    \end{enumerate}
    From~(\ref{it:s9}) and~(\ref{it:s10b}), following the argument in Claim~\ref{cl:2}, we have \(\Test(\eta;b,w)=1\) for each \(\eta\in \cf(\alpha_4)\). Hence, we have \(\bD(\iota,b)\subseteq \alpha_4\). In particular, \(\bD(\iota,b)\neq \iota^-\) and therefore \(\iota\) will not enumerate \(u\) into \(U\). Hence we have \(\bigsame(U;\gamma_{\alpha_0}(BU;x_{(n_j,U)})[s],s,t)\). By~(\ref{it:s3}), we have \(x_{(n_j,U)}>y_{(n,V)}\). Hence we have \(\bigsame(U;\gamma_{\alpha_0}(BU;y_{(n,V)})[s],s,t)\). This completes the proof of Case~2(b) and the proof of this lemma.
\end{proof}

\subsection{Proof of the main theorem}\label{sec:proof of main}
The goal is to show that \(\cT^*\) has an infinite path. The following notations will be convenient in the proof.
For a global parameter \(x\), we say that \(x\) is \emph{stable} after stage~\(s\) if for each \(t\ge s\), \(x[t]=x[s]\). A pair \((\xi\concat n, \beta)\) is \emph{stable} after stage~\(s\) if it will not get canceled after stage~\(s\). If \(x_{(n,U)}\) is stable after \(s\), then \(\gamma(BU;x_{(n,U)})\) is \emph{stable} after \(s_0>s\) if for each \(t>s_0\), \(\gamma(BU;x_{(n,U)})[t]=\gamma(BU;x_{(n,U)})[s_0]\). \(\gamma(BV;y_{(n,V)})\) or \(\gamma(BW;z_{(n,X)})\) being \emph{stable} are defined similarly. A set \(X\res l\) is \emph{stable} after \(s_0\) if for each \(s>s_0\) we have \(\bigsame(X,l,s_0,s)\). 

By Lemma~\ref{lem:R correct} and~\ref{lem:S correct}, we have the following
\begin{lemma}\label{lem:T star}
    Let \(\alpha\in \cT^*\) be a node.
    \begin{enumerate}
        \item Suppose that \(\alpha\) is a \(\Permit(X)\)\nbd{}node. Then \(X\le_T A\).

        \item Suppose that \(\alpha\) is a \(G_e(X)\)\nbd{}node and \(\alpha\concat 0\in T^*\). Then \(\Gamma_e(BX)\neq A\).
        
        \item Suppose that \(\alpha\) is a \(P\)\nbd{}node. Then \(A\) is recursive.
        
        \item Suppose that \(\alpha\) is a \(Q\)\nbd{}node. Then \(A\) is recursive.
        
        \item Suppose that \(\alpha\) is a \(C(\alpha_0,\alpha_1,\alpha_2)\)\nbd{}node and \(\alpha\concat n\in \cT^*\). Then \(\Gamma_{\alpha_0}(BU;x_{(n,U)})\uparrow\).
        
        \item Suppose that \(\alpha\) is an \(R(\alpha_0,\alpha_1,\alpha_2)\)\nbd{}node. Let \(V=V_{\alpha_0}\) and \(W=W_{\alpha_0,\alpha_1}\). If \((n,X)\) is the leftmost outcome that is visited infinitely often, then \(\Gamma_{\alpha_i}(BX;x_{(n,X)})\uparrow\) where \(i=0,1,2\) if \(X=U,V,W\) respectively; if all outcomes are visited finitely often, then \(A\le_T B\).
        
        \item Suppose that \(\alpha\) is an \(R^-(\alpha_0,\alpha_1,\alpha_2)\)\nbd{}node. Let \(V=V_{\alpha_0}\) and \(W=W_{\alpha_0,\alpha_1}\). 
        If \((n,X)\) is the leftmost outcome that is visited infinitely often, then \(\Gamma_{\alpha_i}(BX;x_{(n,X)})\uparrow\) where \(i= 1,2\) if \(X= V,W\) respectively; 
        if all outcomes are visited finitely often, then \(A\le_T B\).
        
        \item Suppose that \(\alpha\) is an \(S(\alpha_0,\alpha_1,\alpha_2)\)\nbd{}node. Let \(V=V_{\alpha_0}\). If \((n,V)\) is the leftmost outcome that is visited infinitely often, then \(\Gamma_{\alpha_1}(BV;y_{(n,V)})\uparrow\); if all outcomes are visited finitely often, then \(A\le_T B\).
        
        \item Suppose that \(\alpha\) is a \(D(X)\)\nbd{}node. If the leftmost outcome that is visited infinitely often is a \(d\)\nbd{}outcome or a \(w\)\nbd{}outcome, then \(\Delta_\alpha\neq X\). \qed
    \end{enumerate} 
\end{lemma}

\begin{lemma}\label{lem:T star C Q}
    Suppose that \(\xi\) is a \(C(\alpha_0,\alpha_1,\alpha_2)\)\nbd{}node such that \(\xi\in \cT^*\) and \(\xi\concat n\notin \cT^*\) for each \(n\)\@. Suppose that \(\beta\) is a \(Q\)\nbd{}node such that \(\xi\concat \omega \subseteq \beta\) and \(\beta\) is initialized finitely often. Then \(\beta\) is disarmed finitely often.
\end{lemma}

\begin{proof}
Let \(V=V_{\alpha_0}\).
Let \(s_0\) be the last stage when \(\beta\) is initialized (and \(\xi\) is not going to be initialized since \(\xi\in \cT^*\)). Suppose that \(\beta\) is not visited after \(s_0\), it is not paired and therefore it is not going be disarmed. Therefore we assume that \(\beta\) is visited after \(s_0\) and \((\xi\concat n,\beta)\) is the stable pair.
Consider \((\xi\concat n,\alpha\concat l,\beta)\)\nbd{}DisarmingProcess for some \(S\)\nbd{}node \(\alpha\). 
We show that \((\xi\concat n,\alpha\concat l,\beta)\)-DisarmingProcess can be triggered finitely often. 

Note first that since \(\beta\) is not initialized after \(s_0\), we are not going to visit any node \(\gamma\) extending \(\xi\concat \omega\) with \(\gamma\prec \beta\). Therefore the edge parameter \(y_{(l,V)}\) is stable for \(\alpha\).

\textbf{Case 1.} Suppose that for each \(l\le n\), \(\xi\concat l\) is visited finitely often. Then~\ref{it:DP-1} and~\ref{it:DP-3} are triggered finitely often. Since we also have \(x_{(n,U)}\) and \(\gamma_{\alpha_0}(BU;x_{(n,U)})\) stable,~\ref{it:DP-5},~\ref{it:DP-2} and~\ref{it:DP-4} are triggered finitely often.

\textbf{Case 2.} Suppose that \(l\le n\) is the least such that \(\xi\concat l\) is visited infinitely often toward a contradiction. Since \(\xi\concat l\notin\cT^*\), \(\xi\concat l\) is initialized infinitely often. Let \(k\le l\) be the least such that \(\xi\concat k\) is initialized infinitely often. As for each \(m<k\), \(\xi\concat m\) is visited finitely often, only~\ref{it:p2} or~\ref{it:p3} are possible to initialize \(\xi\concat k\) infinitely often. Suppose that~\ref{it:p2} happens after \(s_0\). If \(k<n\), then \(\beta\) would be initialized, contradicting the choice of \(s_0\); if \(k=n\), then \((\xi\concat n, \beta)\) must have been canceled, contradicting that \((\xi\concat n,\beta)\) is a stable pair. We are left with the possibility that~\ref{it:p3} happens infinitely often. We assume that \((\xi\concat (k-1),\beta')\) is the pair that triggers~\ref{it:p3} infinitely often. If \(\beta'\prec \beta\), then \(\beta'\) announcing progress would initialize \(\beta\), contradicting the choice of \(s_0\); if \(\beta'=\beta\), then \(k-1=n\), contradicting \(k\le l\le n\). 
\end{proof}

\begin{lemma}\label{lem:T star C alpha}
    Suppose that \(\xi\) is a \(C(\alpha_0,\alpha_1,\alpha_2)\)\nbd{}node such that \(\xi\in \cT^*\) and \(\xi\concat n\notin \cT^*\) for each \(n\)\@. Suppose that \(\alpha\in \cT^*\) is a node such that \(\xi\concat \omega \subseteq \alpha\) and \(\alpha\) is not a \(Q\)\nbd{}node and \(o\) is the leftmost outcome that is visited infinitely often. Then \(\alpha\concat o\in \cT^*\)\@. 
\end{lemma}

\begin{proof}
    
    Let \(s_0\) be the least stage after which \(\alpha\) is not initialized (as \(\alpha\in \cT^*\)). Let \(s_1\ge s_0\) be the least stage after which nodes (in \(\cT[\xi,\Pi_3]\)) to the left of \(\alpha\concat o\)
    \begin{enumerate}
        \item are never visited (by the choice of~\(o\)),
        \item never become activated,
        \item never receive attention (as the definitions of \(\Theta\)\nbd{}functionals of \(Q\)\nbd{}node are not going be extended), 
        \item are never initialized, and
        \item are never disarmed (by Lemma~\ref{lem:T star C Q} and~(4) above).
    \end{enumerate}
    
    Referring to Definition~\ref{def:announces progress}, we conclude that nodes (in \(\cT[\xi,\Pi_3]\)) to the left of \(\alpha\concat o\) will not announce progress after~\(s_1\) and therefore \(\alpha\concat o\) will not be initialized after~\(s_1\). 
    If \(\alpha\concat o\) is not a \(Q\)\nbd{}node, then we immediately have \(\alpha\concat o\in \cT^*\);
    if \(\alpha\concat o\) is a \(Q\)\nbd{}node, we use Lemma~\ref{lem:T star C Q} to conclude that \(\alpha\concat o\) is disarmed only finitely often and hence \(\alpha\concat o\in \cT^*\). 
    
    This completes the proof.
\end{proof}

\begin{lemma}\label{lem:T star path}
    Suppose \(A\nleq_T B\), then \(\cT^*\) has an infinite path.
\end{lemma}

\begin{proof}
    We recursively define a path \(\rho\subseteq \cT^*\). We enumerate the root \(\lambda\) into \(\rho\). Suppose that we have enumerated \(\alpha\) into \(\rho\). Let \(o\) be the leftmost path that is visited infinitely often (whose existence is guaranteed by Lemma~\ref{lem:T star} when applicable), then \(\alpha\concat o\in \cT^*\) (by Lemma~\ref{lem:T star C alpha} when applicable) and we enumerate \(\alpha\concat o\) into \(\rho\). 
\end{proof}

\begin{proof} (of Theorem~\ref{thm:main})
Let \(\rho\) be an infinite path of \(\cT^*\) given by Lemma~\ref{lem:T star path}. Let \(X\) be the set given by Lemma~\ref{lem:full requirement}. By Lemma~\ref{lem:T star}, \(\deg X\) is the desired degree.
\end{proof}

\begin{remark*}
     \(\cT^*\) can possibly have countably many infinite paths and each infinite path of \(\cT^*\) gives us a successful candidate for Theorem~\ref{thm:main}! Having exhibited a successful candidate \(X\), we know that all nonrecursive sets that are Turing reducible to \(X\) are also successful candidates for our Theorem~\ref{thm:main}.
\end{remark*}

\section{Finding the index of a solution}\label{sec:index}
\begin{theorem}\label{lem:index}
    Given r.e.\ sets \(A=W_a\) and \(B=W_b\) with \(A\nleq_T B\), there is a function \(f\le_T 0^{(4)}\) such that \(0<_T W_{f(a,b)}\le_T A\) and \(A\nleq_T B\oplus W_{f(a,b)}\)\@.
\end{theorem}
\begin{proof}
    Let \(\rho\) be the true path (Definition~\ref{def:T star}). We use \(0^{(4)}\) to decide if \(A\le_T B\) or not. If \(A\le_T B\), we define \(f(a,b)=0\). If \(A\nleq_T B\), we will use \(0^{(4)}\) to decide which set is \emph{good} (as defined in the first line of the proof of Lemma~\ref{lem:full requirement}).
    
    First of all, the jump of \(\cT^*\) and hence \(0'''\) can compute the true path~\(\rho\). 
    Next, we use \(\rho'\le_T 0^{(4)}\) to decide whether there is a \(C\)\nbd{}node~\(\xi\) with \(\xi\concat \omega\in \rho\).
    Note that there can be at most one such \(C\)\nbd{}node. Case (i), such \(C\)\nbd{}node~\(\xi\) exists. With~\(\xi\) as a parameter, \(\rho\) becomes \(\le_T 0''\). We start with \(\alpha=\xi\concat \omega\) and use \(\rho''\le_T 0^{(4)}\) to decide whether there are infinitely many nodes \(\beta\in \rho\) such that \(\type(\beta)=\type(\alpha)\). If yes, define \(f(a,b)\) to be the index of~\(X\); if not, we proceed to the next node in~\(\rho\) and repeat. This procedure terminates by Lemma~\ref{lem:full requirement}.
    Case (ii), Such \(C\)\nbd{}node~\(\xi\) does not exist. By a similar argument, the index of the good set can be found using~\(0^{(4)}\).
\end{proof}

In fact, with a very short proof, we are able to exhibit such an~\(f\) without referring to a concrete construction.
\begin{proof}[Another proof.]
    We write \[
        \cS_{a,b}=\{e\mid 0<_T W_e\le_T W_a \land W_a\nleq_T W_b\oplus W_e\}.
    \]
    We have shown in Theorem~\ref{thm:main} that
    \[
        \forall a,b(W_a\nleq_T W_b \iff \cS_{a,b}\neq \varnothing).
    \]
    As \(\cS_{a,b}\le_T 0^{(4)}\) we can define the \(0^{(4)}\)\nbd{}recursive function \[
        f(a,b)=\begin{cases}
        0, & W_a\leq_T W_b;\\
        \min \cS_{a,b}, & W_a\nleq_T W_b. 
        \end{cases}\qedhere
    \]
\end{proof}

\begin{remark*}
In general, deciding whether a set like \(\cS_{a,b}\) is empty requires \(0^{(5)}\); what we have showed is that this is in fact a \(0^{(4)}\) question.
\end{remark*}

Next we prove that no function $f\leq_T 0'''$ can compute the correct index. Consequently, high level of non-uniformity is necessary for the construction.

\begin{theorem}\label{thm:3}
Given $f\leq_T 0'''$, there are two r.e.\ sets $A=W_a$ and $B=W_b$ with \(W_a\nleq W_b\) such that $W_{f(a,b)}\nleq_T A$ or $W_{f(a,b)}\oplus B\geq_T A$\@(i.e., \(f(a,b)\notin \cS_{a,b}\)\@).
\end{theorem}

Let us first recall from~\cite{Barmpalias.Cai.ea:2015} of what can be done and what can not be done. In~\cite{Barmpalias.Cai.ea:2015}, they were trying to build two r.e.\ sets \(A=W_a\) and \(B=W_b\) so that \(A\nleq_T B\) and for all \(e\) we have \(e\notin \cS_{a,b}\). In fact, they developed a successful technique so that for any fixed \(e_0, e_1\), they could build two r.e.\ sets \(A=W_a\) and \(B=W_b\) with \(A\nleq_T B\) so that \(e_0,e_1\notin \cS_{a,b}\).
The conflicts between three and more sets, however, were overlooked.
Fortunately we do not need their technique here to exhibit Theorem~\ref{thm:3}.
We only need the fact that for each fixed \(e\) we can build \(A=W_a\) and \(B=W_b\) with \(A\nleq_T B\) and \(e\notin \cS_{a,b}\). This construction can be done by a simple priority argument.
The exact details are not critical in proving Theorem~\ref{thm:3}; we just need the construction to be \emph{simple enough}. 

This paragraph serves as a reminder for readers who are interested in the construction.
Fixing an index \(e\), we are building two r.e.\ sets \(A\) and \(B\) so that the following requirements for each \(i\in \omega\) are satisfied: 
\begin{itemize}
    \item \(N_i: A\neq \Psi_i(B)\);
    \item \(R_{i,e}: W_e=\Phi_i(A)\to [\exists \Gamma (A=\Gamma(BW_e)) \lor \exists \Delta(W_e=\Delta)]\). 
\end{itemize}
The conflict is between the \(\Gamma\)\nbd{}functionals and \(N\)\nbd{}nodes. When an \(N\)\nbd{}node wants to enumerate a diagonalizing point, say \(a\), into \(A\), it enumerates the point. 
Then an expansionary stage of \(R\)\nbd{}node will tell us whether there is a small change in \(W_e\). If there is one, then \(\Gamma(BW_e;a)\) is undefined without a \(B\)\nbd{}change (\(N\)\nbd{}node is happy) and we redefine it with a fresh use; otherwise, \(N\)\nbd{}node loses its diagonalizing point but it can extend the definition of \(\Delta\). 
The construction has no more conflicts as \(W_e\) is the same set throughout the construction and hence one particular \(\Delta=W_e\) wins all \(R_{i,e}\)\nbd{}requirements for all \(i\in \omega\); this feature is indispensable for this construction.
This construction is referred as \emph{basic} construction with parameter \(e\) and the priority tree will be denoted by \(\cT_e\).

Now let us do the first warm-up to Theorem~\ref{thm:3} by assuming the given \(f\) recursive. 
\begin{proof}[Sketch of the proof for \(f\le_T 0\).]
We are building two sets \(A\) and \(B\). By recursion theorem we know the indices \(a\) and \(b\) for \(A\) and \(B\) to be constructed. Then we do the basic construction with parameter \(f(a,b)\) to get \(A=W_a\) and \(B=W_b\) with \(W_a\nleq_T W_b\) and \(f(a,b)\notin \cS_{a,b}\).
\end{proof}

As we can see from the proof, we simply apply the basic construction with the \emph{correct} parameter to construct \(A\) and \(B\). If \(f\) is recursive, we immediately compute \(f(a,b)\) for the correct parameter. 
However, if \(f\) is not recursive, we have to guess the value of \(f(a,b)\). Sometimes we run the basic construction with a \emph{wrong} parameter, sometimes with the correct one. We have to ensure that the basic construction with correct parameter is not interrupted by those with wrong ones: the idea is that the basic construction is fairly simple and the conflicts between basic constructions with different parameters are not fatal. Now we do the second warm-up for the double jump (we skip the warm-up for the jump). 

\begin{proof}[Sketch of the proof for \(f\le_T 0''\).]
Using recursion theorem we know the indices \(a\) and \(b\) for the sets we are constructing. We have to guess what \(f(a,b)\) is. Since \(f\le_T 0''\), so the relation \(f(a,b)=e\) is both \(\Sigma_3\) and \(\Pi_3\). Let \(R\) and \(S\) be recursive predicates such that \[
    f(a,b)=e \iff \exists u [R(u,e) \text{ happens infinitely often}]  
\] and \[ 
    f(a,b)\neq e\iff \exists u [S(u,e) \text{ happens infinitely often}],
\]
where \(R(u,e)=R(u,e,a,b)\) and \(S(u,e)=S(u,e,a,b)\) and we suppress the parameters \(a\) and \(b\).


Then we define the \emph{guessing tree} \(\cG\) to guide our construction. The root \(\lambda\) is assigned \(G(0)\). A \(G(e)\)\nbd{}node \(\alpha\) has \(\omega\) outcomes ordered by \(0< 1<\cdots<2u<2u+1<\cdots\). Each \(\alpha\concat (2u)\) is a terminal node and \(\alpha\concat (2u+1)\) is assigned \(G(e+1)\).

At each stage~\(s\), we begin with \(\visit(\lambda)\), where \(\visit(\alpha)\) has the following strategy:

Suppose that \(\alpha\) is an \(G(e)\)\nbd{}node.
\begin{enumerate}
    \item Let \(u\) be the least, if any, such that \(R(u,e)\) happens or \(S(u,e)\) happens. 
    \item If \(R(u,e)\) happens, \(\visit(\alpha\concat (2u))\).
    \item If \(S(u,e)\) happens, \(\visit(\alpha\concat (2u+1))\).
\end{enumerate}

Suppose that \(\alpha\) is \(\beta\concat (2u)\) for some \(G(e)\)\nbd{}node \(\beta\). We \(\visit(\lambda)\) for the root \(\lambda\) in \(\cT_{e}(u)\), where \(\cT_{e}(u)\) is a copy of the priority tree for the basic construction with parameter \(e\) and \(u\) is a parameter from \(\cG\). (\emph{We think of the terminal node of \(\cG\) as an entrance to the priority tree for the basic construction with corresponding parameters.})

We have the usual left-to-right initializations: if \(\beta\concat (2u)\) is initialized (because a node to the left is visited), then all nodes of \(\cT_{e}(u)\) are initialized.

The true path \(\rho\) of \(\cG\) is the leftmost path that is visited infinitely often. If \(G(e)\concat (2u+1)\in \rho\), then \(f(a,b)\neq e\); if \(G(e)\concat (2u)\in \rho\), then \(f(a,b)=e\). Therefore \(\rho\) is finite. Let \(G(e)\concat (2u)\) be the longest node of \(\rho\), then \(\cT_{e}(u)\) is visited infinitely often. Note that each node along the true path of \(\cT_{e}(u)\) is initialized still only finitely often.
\end{proof}

If \(f\le_T 0''\), \(\cT_{e}(u)\) with different \(e\) and \(u\) have no conflicts other than left-to-right initializations. This will not be the case when \(f\le_T 0'''\). We present the main ideas to overcome the conflicts and to avoid overburdening the readers with technical details.

\begin{proof}[Sketch of the proof for Theorem~\ref{thm:3}]
Using the limit lemma, we know \[
    e=f(a,b)=\lim_{z\to \infty} g(a,b,z)
\] for some \(g\le_T 0''\), therefore \(g(a,b,z)=e\) can be guessed using the same process as in the above proof. 


We now define our guessing tree \(\cG\): the root \(\lambda\) is assigned \(G(0,0)\). Suppose that \(\alpha\) is assigned \(G(z,e)\), it has \(\omega\) outcomes ordered as \(0<1<\cdots<2u<2u+1<\cdots\). \(\alpha\concat (2u)\) is assigned \(G(z+1,0)\) and \(\alpha\concat (2u+1)\) is assigned \(G(z,e+1)\).
If \(\rho\) is the true path along \(\cG\), we have that \(G(z,e)\concat (2u)\in \rho\) for some \(u\) if and only if \(g(a,b,z)=e\). 

As we travel along the guessing tree \(\cG\) and visit a node \(G(z,e)\concat (2u)\), we better allow the basic construction with parameter \(e\) to run. But as our guessing tree \(\cG\) is more complicated and two nodes \(\alpha\subsetneq \beta\) might have guessed for different parameters \(e_\alpha\neq e_\beta\), these two basic constructions with parameters \(e_\alpha\) and \(e_\beta\) have to find ways to cooperate.

The \emph{predecessor} of \(\alpha\) is the longest node \(\beta\), if any, such that \(\beta\concat(2u)\subseteq \alpha\) for some \(u\).
A \(G(z,e)\)\nbd{}node \(\alpha\) is an \emph{entrance} to the basic construction with parameter \(e\) if its predecessor is not a \(G(z-1,e)\)\nbd{}node. 
For \(\alpha\in \cG\), we let \[
    \alpha_0\concat (2u_0)\subseteq\alpha_1\concat (2u_1)\subseteq\cdots \alpha_{k-1}\concat (2u_{k-1})\subseteq \alpha_k=\alpha
\] 
be the sequence such that \(\alpha_0\) is an entrance for the basic construction with parameter \(e\) and \(\alpha_i\) is the predecessor of \(\alpha_{i+1}\) and suppose that \(\alpha_i\) is a \(G(z_0+i,e)\)\nbd{}node (so \(\alpha\) is a \(G(z_0+k,e)\)\nbd{}node). 
For each outcome \((2u)\) of \(\alpha\), we let \(\cT_{e,\alpha_0}(u_0,u_1,\dots,u_{k-1},u)\) extend \(\cT_{e,\alpha_0}(u_0,u_1,\dots,u_{k-1})\) by copying the nodes in \(\cT_e\) of length \(k\), where \(\cT_e\) is the priority tree for the basic construction with parameter \(e\). The idea is that at \(\alpha\concat (2u)\in \cG\) we are only allowed to visit a node \(\sigma\in \cT_{e,\alpha_0}(u_0,u_1,\dots,u_{k-1},u)\) with length \(k-1\): the algorithm for \(\visit(\sigma)\) is \emph{paused} whenever \(\sigma\in \cT_{e,\alpha_0}(u_0,u_1,\dots,u_{k-1},u)\) has length \(k\). We let \(\sigma_{\alpha,u}\) denote the paused node. Whenever we \(\visit(\alpha\concat (2u))\) in the guessing tree \(\cG\), we simultaneously \(\visit(\sigma_{\alpha_{k-1},u_{k-1}})\) for the node \( \sigma_{\alpha_{k-1},u_{k-1}}\in \cT_{e,\alpha_0}(u_0,u_1,\dots,u_{k-1},u)\).

The initialization is the usual left-to-right kind: if we are visiting a node to the left of \(\alpha\concat (2u)\), then the nodes in \(\cT_{e,\alpha_0}(u_0,\dots,u_{k-1},u)\setminus \cT_{e,\alpha_0}(u_0,\dots,u_{k-1})\) are initialized. 

To avoid conflicts between different basic constructions, each entrance \(\alpha\) picks a fresh point \(w_u\) for each \(u\in\omega\). Whenever we \(\visit(\alpha\concat (2u))\) for some \(u\), we enumerate \(\gamma(w_u)\) into \(B\) to undefine all the \(\Gamma\)\nbd{}functionals built by some \(\beta\concat (2u')\subseteq \alpha\) for some \(u'\).

By a routine verification, one derives the following:
\begin{enumerate}
    \item The true path \(\rho\) of \(\cG\) is infinite and there is an infinite sequence \[
        \alpha_0\concat (2u_0)\subseteq \alpha_1\concat (2u_1)\subseteq \cdots
    \]
    such that \(\alpha_i\in \rho\), \(\alpha_0\) is an entrance for the basic construction with parameter \(e=f(a,b)\), \(\alpha_i\) is a \(G(z_0+i,e)\)\nbd{}node for each \(i\) for some \(z_0\), and \(\alpha_i\) is the predecessor of \(\alpha_{i+1}\).
    \item For each \(i\), nodes in the tree \(\cT_{e,\alpha_0}(u_0,\dots, u_i)\) are visited infinitely often and initialized finitely often.
    \item The union of all \(\cT_{e,\alpha_0}(u_0,\dots, u_i)\) is a full copy of \(\cT_e\), the priority tree for the basic construction with parameter \(e\).
\end{enumerate}

This completes the proof.
\end{proof}

\section*{Acknowledgement}
The authors received helpful advice from many peoples during the years when we were working on this paper, and they would like to thank George Barmpalias, Yun Fan, Steffen Lempp, Keng Meng Ng, Theodore Slaman and Liang Yu for valuable discussions and their useful comments.


\begin{thebibliography}{10}

\bibitem{Barmpalias.Cai.ea:2015}
Barmpalias, George;  Cai, Mingzhong;  Lempp, Steffen;  Slaman, Theodore A.  
\newblock On the existence of a strong minimal pair.
\newblock \emph{J. Math. Log.},  15  (2015),  no. 1, 1550003, 28 pp.

\bibitem{Shore:1988}
Shore, Richard A. 
\newblock A noninversion theorem for the jump operator.
\newblock \emph{Ann. Pure Appl. Logic},  40  (1988),  no. 3, 277--303.

\bibitem{Shore.Slaman:1993} Shore, Richard A., and Slaman, Theodore A.
\newblock Working below a high recursively enumerable degree.
\newblock \emph{J. of Sym. Logic}, 58 (1993), pp. 824–859.
\end{thebibliography}
\end{document}